\documentclass[11pt]{amsart}
\usepackage[T1]{fontenc}
\usepackage{txfonts}
\usepackage{times}
\usepackage{amssymb,verbatim}
\usepackage[all]{xy}
\usepackage{subfig}
\usepackage{tikz}
\usepackage{color}
\usepackage{a4wide}
\usepackage{hyperref}
\usepackage{wrapfig}
\usetikzlibrary{arrows}

\newcommand{\R}{\mathbb{R}}
\newcommand{\C}{\mathbb{C}}

\newcommand\Z{\mathbb{Z}}

\newcommand{\N}{\mathbb{N}}

\renewcommand{\H}{\mathcal{H}}
\newcommand{\T}{\mathrm{T}}
\newcommand{\hT}{\hat{\mathrm{T}}}
\newcommand{\SL}{{\rm SL}}
\newcommand{\GL}{{\rm GL}}
\newcommand{\Mod}{\mathfrak{M}}
\newcommand{\Aa}{\mathbf{A}}
\renewcommand{\S}{\mathbb{S}}

\newcommand{\Hg}{\mathcal{H}(\underline{k})}
\newcommand{\Hgi}{\mathcal{H}_1(\underline{k})}
\newcommand{\Hgie}{\mathcal{H}_{1,\underline{\eps}}^{(m)}(\underline{k})}
\newcommand{\Hge}{\mathcal{H}_{\underline{\eps}}^{(m)}(\underline{k})}
\newcommand{\Hgm}{\widetilde{\mathcal{H}}^{(m)}(\underline{k})}
\newcommand{\Hgme}{\widetilde{\mathcal{H}}_{\underline{\eps}}^{(m)}(\underline{k})}

\newcommand{\Hgms}{\widetilde{\mathcal{H}}^{(m)}(\underline{k})^*}
\newcommand{\Qg}{\mathcal{Q}(\underline{d})}
\newcommand{\Qgi}{\mathcal{Q}_1(\underline{d})}

\newcommand{\Qgie}{\mathcal{Q}_{1,\underline{\eps}}^{(m)}(\underline{d})}
\newcommand{\Qgm}{\widetilde{\mathcal{Q}}^{(m)}(\underline{d})}

\newcommand{\Qgms}{\Qgm^*}

\newcommand{\ST}{\mathbf{S}_\mathrm{T}}
\newcommand{\VT}{\mathrm{V}_\mathrm{T}}

\newcommand{\VG}{\mathrm{V}_\G}
\newcommand{\SG}{\mathbf{S}_\G}

\newcommand{\ol}{\overline}
\newcommand{\ul}{\underline}

\newcommand{\ra}{\rightarrow}

\newcommand{\Id}{\mathrm{Id}}
\newcommand{\vide}{\varnothing}

\newcommand{\eps}{\epsilon}
\newcommand{\veps}{\varepsilon}
\newcommand{\inter}{\mathrm{int}}
\newcommand{\Sig}{\Sigma}

\newcommand{\lbd}{\lambdaup}
\newcommand{\g}{\gamma}
\newcommand{\G}{\Gamma}
\newcommand{\Alp}{\mathcal{A}}

\newcommand{\U}{\mathrm{U}}
\newcommand{\V}{\mathrm{V}}

\newcommand{\Dc}{\mathcal{D}}
\newcommand{\Bc}{\mathcal{B}}
\newcommand{\Mm}{\mathcal{M}}
\newcommand{\Sc}{\mathcal S}
\newcommand{\UU}{\mathcal{U}}
\newcommand{\VV}{\mathcal{V}}

\newcommand{\Pb}{\mathbb{P}}

\newcommand{\tp}{\tauup}
\newcommand{\card}{{\rm card}}

\newcommand{\DS}{\displaystyle}

\newcommand{\vol}{{\rm vol}}

\newcommand{\Cc}{\mathcal{C}}
\newcommand{\re}{\Re}
\newcommand{\im}{\Im}
\newcommand{\pp}{\mathbf{p}}
\newcommand{\Cb}{\mathbf{C}}
\newcommand{\Sys}{\varrho}
\newcommand{\Ic}{\mathcal{I}}
\newcommand{\Ub}{\mathbf{U}}

\newtheorem{Theorem}{Theorem}[section]
\newtheorem{Corollary}[Theorem]{Corollary}
\newtheorem{Lemma}[Theorem]{Lemma}
\newtheorem{Proposition}[Theorem]{Proposition}
\newtheorem{Remark}[Theorem]{Remark}
\newtheorem{Definition}[Theorem]{Definition}

\newtheorem{Claim}[Theorem]{Claim}
\title[Surfaces with small saddle connections in rank one affine submanifolds]{Volumes of the sets of translation surfaces with small saddle connections in rank one affine submanifolds}
\author{Duc-Manh Nguyen}

\address{IMB Bordeaux,\newline
CNRS UMR 5251\newline
Universit\'e de Bordeaux \newline
351, Cours de la Lib\'eration \newline
33405 Talence \newline
FRANCE}

\email{duc-manh.nguyen@math.u-bordeaux.fr}
\date{\today}

\begin{document}
\begin{abstract}
We prove some estimates of the volumes of the sets of translation surfaces of unit area having several independent small saddle connections in a rank one affine submanifold.
\end{abstract}

\maketitle

\section{Introduction}
Translation surfaces are flat surfaces defined by holomorphic $1$-forms on compact Riemann surfaces. The space of translation surfaces of genus $g$ is parametrized by the Hodge bundle $\Omega \Mod_g$ over the moduli space $\Mod_g$ of Riemann surfaces of the same genus.
For each integral vector $\ul{k}=(k_1,\dots,k_n)$ such that $k_i >0$ and $k_1+\dots+k_n=2g-2$,
we denote by $\H(\ul{k})$ the set of translation surfaces defined by holomorphic $1$-forms which have exactly $n$ zeros with orders $k_1,\dots,k_n$.
The set $\H(\ul{k})$ is called a {\em stratum} of $\Omega\Mod_g$.
It is well-known that $\H(\ul{k})$ is an algebraic orbifold of complex dimension $2g+n-1$, and that $\H(\ul{k})$ can be locally identified with $H^1(M,\Sig,\C)$ via the period mappings, where $M$ is a base surface in $\H(\ul{k})$ and $\Sig$ is its set of singularities. For some  introductions to the subject, we refer to \cite{MasTab,Zor06, Wright_survey1}.

Let $\H_1(\ul{k})$ denote the set of surfaces in $\H(\ul{k})$ which have unit area.
There is a natural action of $\GL^+(2,\R)$ on $\H(\ul{k})$ such that the subgroup $\SL(2,\R)$ preserves $\H_1(\ul{k})$.
It turns out that the geometric and dynamical properties of a translation surface are often encoded in its $\GL^+(2,\R)$-orbit closure.
The classification of those orbit closures is a central problem of the field.

Globally, this problem has been solved by the works~\cite{EsMir12, EsMirMoh15} in which it is shown that every $\GL^+(2,\R)$-orbit closure is an immersed submanifold $\Mm$ of $\H(\ul{k})$, which is locally identified with some linear subspace $V$ of $H^1(M,\Sig,\C)$ defined by linear equations with real coefficients. Moreover, $\Mm$ carries a volume form $\vol$ proportional to the Lebesgue measure of $V$,  which induces an ergodic $\SL(2,\R)$-invariant probability measure $\vol_1$ on the set $\Mm_1$ of surfaces in $\Mm$ of unit area. Such submanifolds are called {\em invariant affine submanifolds} (or {\em affine submanifolds}) of $\H(\ul{k})$.

Since $V$ is defined by linear equations with real coefficients, we can write $V=V^\R\oplus \imath V^\R$, where $V^\R:=V\cap H^1(M,\Sig,\R)$.
Let $\pp: H^1(M,\Sig,\R) \ra H^1(M,\R)$ be the natural projection. By a result of \cite{AvEskMo12}, the restriction  of the intersection form of $H^1(M,\R)$ to $\pp(V^\R)$ is non-degenerate.
Therefore $\dim_\R \pp(V^\R)=2r$, with $r \in \{1,\dots,g\}$. The number $r$ is call the {\em rank} of $\Mm$.

In applications, for instance the computation of the Siegel-Veech constants (see~\cite{EskMasZ, MasZor}), or the computation of the sum of the Lyapunov exponents for the Teichm\"uller geodesic flow (see~\cite{EskKonZor11}), it is important to have an estimate on the volume of the set of surfaces having several non-parallel saddle connections in a given affine submanifold.
In some sense, such  estimates provide information about the ``regularity'' of the affine submanifold near its boundary.
In \cite{EskKonZor11}, Eskin-Kontsevich-Zorich define a notion of regularity for invariant affine submanifolds as follows:
given $\Mm$ as above, for any $K>0, \eps>0$, let $\Mm_1(K,\eps)$ denote the set of surfaces in $\Mm_1$ which have two non-parallel cylinders $C_1,C_2$ satisfying $\mathrm{mod}(C_i) >K$ and $\ell(C_i) < \eps$, where $\mathrm{mod}(C_i)$ and $\ell(C_i)$ are the modulus and the circumference  of $C_i$ respectively.
Then $\Mm$ is said to be {\em regular} if there exists $K>0$ such that
$$
\lim_{\eps \ra 0} \frac{\vol_1(\Mm_1(K,\eps))}{\eps^2}=0.
$$
That strata are regular is known by the work of Masur-Smillie~\cite{MasSmi91}.
In \cite{AviMatYoc13}, Avila-Matheus-Yoccoz show that  every affine submanifold of $\H(\ul{k})$ is actually regular.
In fact the result of \cite{AviMatYoc13}  is  that $\Mm$ satisfies  a slightly stronger condition.
Namely, let $\Mm_1(\eps^2)$ denote the set of surfaces in $\Mm_1$ having two non-parallel saddle connections of length smaller than $\eps$,  then for small $\eps >0$, we have
$$
\vol_1(\Mm_1(\eps^2))= o(\eps^2).
$$
In this paper, we will be focusing on rank one affine submanifolds of strata.
This class of affine submanifolds includes the close  orbits (Teichm\"uller curves) and the Prym eigenform loci discovered by McMullen~\cite{McM06,McM07}.
We will show that such affine submanifolds have the same  ``regularity'' as strata, thus improve the result of \cite{AviMatYoc13} in this case.

Specifically, let $\Mm$ be a $d$-dimensional rank one affine submanifold of a stratum $\H(\ul{k})$.
For any positive integer $\nu$ such that $\nu < d$, and any $\eps >0$, let $\Mm(\eps^\nu)$ be the set of surfaces $M \in \Mm$ such that $M$ contains $\nu$ saddle connections $e_1,\dots,e_\nu$ satisfying
\begin{itemize}
\item[(i)] $\max\{|e_1|,\dots,|e_\nu|\} < \eps \sqrt{\Aa(M)}$,  where $|e_i|$ is the euclidian length of $e_i$, and $\Aa$ is the area function,

\item[(ii)] $\{e_1,\dots,e_\nu\}$ is an independent family in $(T_M\Mm)^*$.
\end{itemize}
Let $\Mm_1(\eps^\nu)$ denote the intersection of $\Mm(\eps^\nu)$ with $\H_1(\ul{k})$. We will show

\begin{Theorem}\label{thm:estimate:eps:rk1}
There exist some constants $\eps_0>0$ and $K_0>0$ such that for any  $0< \eps < \eps_0$, we have
\begin{equation}\label{eq:main:integral}
\int_{\Mm(\eps^\nu)} e^{-\Aa}d\vol < K_0\eps^{2\nu}.
\end{equation}
Equivalently, there exists a constant $\tilde{K}_0>0$, such that $\vol_1(\Mm_1(\eps^\nu)) < \tilde{K}_0\eps^{2\nu}$, for any $0< \eps < \eps_0$.
\end{Theorem}

To obtain this result, our strategy is to specify a finite family of local charts, associated with cylinder decompositions of surfaces in $\Mm$, that cover an open subset of full measure in $\Mm$ (see Theorem~\ref{thm:good:subsp:finite:G}).  We then show that the integral of the function $e^{-\Aa}$ over the intersection of $\Mm(\eps^\nu)$ with the domain of any local chart in this family satisfies the inequality \eqref{eq:main:integral} (see Proposition~\ref{prop:estimate:kap:eps:1}, \ref{prop:estimate:kap:eps:2}).
The finiteness of the local charts associated with cylinder decompositions is derived from the fact that $\Mm$ is a quasiprojective  subvariety of $\H(\ul{k})$, which is shown in \cite{Fil:algebraic}.


In the appendices, we prove some estimates for the volumes of the set of surfaces having several small saddle connections in strata of Abelian differentials and quadratic differentials by using similar ideas  to the proof of Theorem~\ref{thm:estimate:eps:rk1}. Those estimates were actually known by the work of Masur-Smillie~\cite{MasSmi91}.

\subsection*{Acknowledgements:}
The author warmly thanks Vincent Koziarz for the helpful discussions. He thanks Alex Wright for some useful comments on an earlier version of this paper.

\section{Preliminaries on rank one affine manifolds}\label{sec:cyl:dec:charts}
\subsection{Complete periodicity}
Let us fix a surface $M=(X,\omega) \in \Mm$.  We denote by $\Sig$ the set of singularities of $M$ (zeros of $\omega$).
Let $V$ denote the  tangent space $T_M\Mm \subset H^1(M,\Sig,\C)$ of $\Mm$ at $M$.  
We will consider a cycle  $c \in H_1(M,\Sig,\Z)$ as an element of $(H^1(M,\Sig,\C))^*$, and denote by $c_V$ the restriction of $c$ to $V$.
Two cycles $c_1$ and $c_2$ of $H_1(M,\Sigma,\Z)$ are said to be $\Mm$-parallel if there is real constant $\lbd$ such that $(c_1)_V=\lbd (c_2)_V$, that is $T_M\Mm \subset \ker(c_1-\lbd c_2) \subset H^1(M,\Sigma,\C)$.

Since $\Mm$ is of rank one, if $M'\in \Mm$ is close enough to $M$, there exist a matrix  $A \in \GL^+(2,\R)$ close to $\Id$, and  a vector $v \in \ker(\pp)\cap V$ close to $0$ such that we can write $M'=A \cdot(M+v)$ (see {\em e.g.} \cite[Cor.3.2]{Lanneau:Manh:cp} or \cite[Prop.2.6]{Lan-Ng-Wr15}).

By the work of Wright~\cite{Wr13}, we know that $M$  is completely periodic in the sense of Calta, that is the direction of any cylinder is periodic (see~\cite{Cal04}). As noticed by A.Wright (see~\cite[Rem. 4.21]{Wright_survey1}), with the arguments in \cite{Wr13}, one can actually show more. Namely, we have

\begin{Theorem}[Wright]
\label{thm:cp:stronger}
Assume that there exists a family of horizontal saddle connections of $M$ which form a cycle in $H_1(M,\Z)$ with non-zero holonomy. Then $M$ is horizontally periodic.
\end{Theorem}

\begin{Definition}\label{def:stable:periodic}
The surface $M \in \Mm$ is said to be $\Mm$-stably periodic in  direction $\theta$ if it is periodic in that direction, and all the saddle connections in this direction remain parallel on any surface in a neighborhood of $M$ in $\Mm$ (see~\cite[Sec. 2.3]{Lan-Ng-Wr15}).
Equivalently, we will say that $M$ admits an $\Mm$-stable cylinder decomposition in direction $\theta$.
\end{Definition}

In the language of \cite{Wr13}, $M$ is $\Mm$-horizontally stably periodic means that ${\rm Twist}(M,\Mm)={\rm Pres}(M,\Mm)$. It also means that for any vector $v \in T_M\Mm\cap \ker(\pp)$ with $||v||$ small enough, $M+v$ is horizontally periodic with the same cylinder diagram as $M$ (compare with~\cite{Lanneau:Manh:cp}). We have
\begin{Lemma}
\label{lm:stable:full:measure}
The set of surfaces having a periodic direction which is not $\Mm$-stable has measure zero in $\Mm$.
\end{Lemma}
\begin{proof}
If $M$ has a periodic direction that is not  $\Mm$-stable then there exist two elements $c_1,c_2$ of $H_1(M,\Sigma,\Z)$, which are not $\Mm$-parallel, such that $M$ is contained in the subset of $V$ defined by  $\{ v \in V \;  | \; \im(\langle v, c_1 \rangle\ol{\langle v,c_2\rangle})=0\}$.
Thus the total measure of  the set of such surfaces (with respect to the volume form $\vol$ of $\Mm$) is zero.
\end{proof}

\subsection{Stable cylinder decompositions and local charts}
Assume now that $M$ is horizontally periodic. Let $C_1,\dots,C_k$ denote the horizontal cylinders, and  $a_1,\dots,a_m$ the horizontal saddle connections of $M$.
Note that the number of horizontal saddle connections $m$ depends only on the stratum $\H(\ul{k})$. Let $\G$ denote the corresponding cylinder diagram.


We will call a saddle connection in $C_j$ which intersects every core curve of $C_j$ once a {\em crossing saddle connection}. Any crossing saddle connection $b$ must join the left endpoint of a horizontal saddle connection $a_{i_1}$ in the bottom of $C_j$ to the left endpoint of a horizontal saddle connection $a_{i_2}$ in the top of $C_j$. We will call $(i_1,i_2)$  index of $b$.
\begin{Definition}\label{def:equiv:cross:sc}
Two crossing saddle connections  are {\em equivalent}  there is a Dehn twist of $C_j$ that sends one to the other.
Equivalently, two crossing saddle connections are  equivalent if they have the same index.
\end{Definition}
Up to a renumbering, we can assume that $\{a_1,\dots,a_{m_0}\}$ is a maximal independent family of horizontal saddle connections in $H_1(M,\Sig,\Z)$.
For every $j\in\{1,\dots,k\}$, we pick a crossing saddle connection $b_j$ in $C_j$.
It is not difficult to check that $\Bc:=\{a_1,\dots,a_{m_0},b_1,\dots,b_k\}$ is a basis of $H_1(M,\Sig,\Z)$.

Set
$$
\U_\G:= \{(r_1,\dots,r_{m_0},x_1,y_1,\dots,x_k,y_k) \in (\R_{>0})^{m_0}\times\R^{2k}\, | \,  y_j>0, j=1,\dots,k\}.
$$
We will  consider $\U_\G$ as a subset of $(\R_{>0})^{m_0}\times\Ub^{k}$, where $\Ub=\{z \in \C, \; \im(z) >0\}$, with the identification $\allowbreak (r_1,\dots,r_{m_0},x_1,y_1,\dots,x_k,y_k) \simeq (r_1,\dots,r_{m_0},x_1+\imath y_1,\dots,x_k+\imath y_k)$.


We have a map $\widetilde{\Phi}: \U_\G  \ra \H(\ul{k})$ defined as follows: given $(r_1,\dots,r_{m_0},z_1,\dots,z_k) \in \U_\G$,  for $i=1,\dots,m_0$, we assign to $a_i$ the length $r_i$,  and  compute the lengths of $a_{m_0+1},\dots,a_m$ from $(r_1,\dots,r_{m_0})$ using the relations in $H_1(M,\Sig,\Z)$. We associate to the crossing saddle connection $b_j$ the complex number $z_j$, for $j=1,\dots,k$. We can then construct the cylinder $C_j$ from a parallelogram in $\R^2$ determined by those data. Finally, the diagram $\G$ provides us with the rules to glue those cylinders together to obtain a surface in $\H(\ul{k})$.

\begin{Lemma}\label{lm:Phi:loc:inject}
The map $\widetilde{\Phi}$ is locally injective and  $\widetilde{\Phi}^{-1}(\Mm)$ consists of the intersections of $\U_\G$ with  a family  of real linear subspaces of dimension at most $2d-1$ in  $\R^{m_0}\times\R^{2k}$. A subspace in this family has dimension $2d-1$ if and only if its image under $\widetilde{\Phi}$ contains a surface which is $\Mm$-stably horizontally periodic with cylinder diagram $\G$.
\end{Lemma}
\begin{proof}
Note that in any local chart of $\H(\ul{k})$ defined by a period mapping, $\widetilde{\Phi}$ is a linear map. Since $\Bc$ is a basis of $H_1(M,\Sig,\Z)$, $\widetilde{\Phi}$ is clearly locally injective.
For $\bar{x} \in \U_\G$ we have
 $$
 d\widetilde{\Phi}(\bar{x})(\R^{m_0}\times\R^k)=H^1(M,\Sig,\R), \text{ and } d\widetilde{\Phi}(\bar{x})(\{0\}\times(\imath\R)^k) \subset H^1(M,\Sig,\imath\R).
 $$
Let $M=\widetilde{\Phi}(\bar{x})\in \Mm$, and  $V:=T_M\Mm$.
We can write $V=V^\R\oplus\imath V^\R$, where $V^\R:=V \cap H^1(M,\Sig,\R)$.
Recall that $\dim_\C \Mm=\dim_\C V=\dim_\R V^\R=d$.
Consider now a surface $M'$ in $\H(\ul{k})$ which is close enough to $M$ such that all the saddle connections $(a_1,\dots,a_m)$ persist on $M'$.
Note that  $M' \in \widetilde{\Phi}(\U_\G)$ if and only if the saddle connections $a_1,\dots,a_m$ are all  horizontal in $M'$.
Identifying $M'$  with a vector $v\in H^1(M,\Sig,\C)$, we see that $M' \in \widetilde{\Phi}(\U_\G)$ if and only if $\im \langle v, a_i\rangle=0$.
This implies that in  a neighborhood of $M$, we have
$$
 \Mm\cap\widetilde{\Phi}(\U_\G)=\{v'+\imath v''\,  | \, v' \in V^\R, \, v'' \in V^\R,  \,  \langle v'', a_i\rangle=0, \, i=1\dots,m\}.
$$
In other words, $T_M\Mm\cap \mathrm{Im}(d\widetilde{\Phi}(x))=V^\R\oplus \imath W^\R$, where
$$
W^\R:=\{v\in V^\R, \,  \langle v, a_i\rangle=0, \, i=1\dots,m\}.
$$
Clearly $\dim_\R W^\R\leq d-1$.  It follows that $\dim_\R V^\R\oplus \imath W^\R \leq 2d-1$.

 Assume now that $M$ is $\Mm$-stably horizontally periodic. We claim that in this case $\dim W^\R=d-1$.  To see this, recall that by the definition, we have
 ${\rm Twist}(M,\Mm)={\rm Pres}(M,\Mm)$.
 Since $\Mm$ is of rank one, this condition means that, if $c_1$ is a core curve of  the cylinder $C_1$,   then there exist some  real positive constants $\lbd_i, \, i=1,\dots,m$, such that $(a_i)_V=\lbd_i \cdot (c_1)_V$,   where $(c_1)_V$ and $(a_i)_V$ are the restrictions of $c_1$ and $a_i$ to $V$ respectively.
 Thus we have $W^\R=V^\R\cap \ker(c_1)$. Since $c_1$ does not vanish identically on $V^\R$, we have $\dim_\R W^\R=\dim_\R V^\R-1=d-1$.

 Conversely, if $\dim_\R W^\R=d-1$, then we must have $\ker(a_i)\cap V^\R=\ker(c_1)\cap V^\R$, which implies that $(a_i)_V$ and $(c_1)_V$ are proportional. Hence $M$ is $\Mm$-stably horizontally periodic.

 Since $d\widetilde{\Phi}(\bar{x}): \R^{m_0}\times \C^k \ra H^1(M,\Sig,\R\oplus\imath\R)$ is an  $\R$-linear injective map,
 the subspace $V_{\bar{x}}:=(d\widetilde{\Phi}(\bar{x}))^{-1}(V^\R\oplus\imath W^\R)$  of $\R^{m_0}\times \C^k$  satisfies  $\dim_\R V_{\bar{x}}=\dim_\R (V^\R\oplus \imath W^\R) \leq 2d-1$,
 and the equality occurs if and only if $M$ is $\Mm$-stably horizontally periodic by the arguments above.
\end{proof}

\begin{Definition}\label{def:alpha}
Let $\{V_\alpha, \, \alpha \in \Alp_\G\}$ denote the family of real linear subspaces of  dimension $2d-1$ in $\R^{m_0}\times \R^{2k}$ such that  $\U_{\G,\alpha}:=V_\alpha  \cap \U_\G$ is non-empty and   $\widetilde{\Phi}(\U_{\G,\alpha}) \subset \Mm$.
\end{Definition}

By Lemma~\ref{lm:Phi:loc:inject}, we know that there is an isomorphism of $\R$-linear spaces between $V_\alpha, \alpha \in \Alp_\G$, and $V^\R\oplus \imath W^\R$,  where $V^\R=T^\R_{\widetilde{\Phi}(\bar{x})}\Mm$ for some $\bar{x} \in \U_\G^\alpha$, and $W^\R$ is a linear subspace of codimension one in $V^\R$.
Using the same notation as in Lemma~\ref{lm:Phi:loc:inject}, we have the following identification $\U_{\G,\alpha}\simeq V_+^\R\times W_+^\R$, where
$$
V^\R_+:=\{v \in V^\R, \; \langle v, c_1 \rangle >0\} \text{ and } W^\R_+:=\{v \in V^\R, \; \langle v, c_1\rangle =0, \langle v, b_j \rangle >0, \, j=1,\dots,k\}.
$$

Since both $V^\R_+$ and $W^\R_+$ are convex (hence connected),   the definition of $\Mm$-stably periodic surfaces implies that the ratio $r_{i_1}/r_{i_2}$ is  constant in $\U_{\G,\alpha}$ for any $i_1,i_2\in \{1,\dots,m_0\}$. Thus we have
\begin{Lemma}\label{lm:r:proportional}
Assume that $\bar{x}=(\bar{r},z_1,\dots,z_k)$ and $\bar{x}'=(\bar{r}',z'_1,\dots,z'_k)$, where $\bar{r},\bar{r}' \in (\R_{>0})^{m_0}$, both belong to $\U_{\G,\alpha}$.  Then there exists a positive real constant $\lbd$ such that $\ol{r}'=\lbd \ol{r}$.
\end{Lemma}

\subsection{Image in the projectivized Hodge bundle}\label{sec:project:closure}
By definition, $\Mm$ is contained in the complement of the zero section of the Hodge bundle $\Omega\Mod_g$ over the moduli space $\Mod_g$.
Let $\ol{\Mod}_g$ denote the Deligne-Mumford compactification of $\Mod_g$. Then  $\Omega\Mod_g$ is the restriction to $\Mod_g$ of the bundle of stable differentials $\Omega\ol{\Mod}_g$ over $\ol{\Mod}_g$. Over a stable curve in the boundary of $\ol{\Mod}_g$, the fiber of $\Omega \ol{\Mod}_g$ consists of meromorphic differentials with  simple poles and opposite residues at the nodes. We denote by $\Pb\Omega\ol{\Mod}_g$ the projectivized bundle of $\Omega\ol{\Mod}_g$ over $\ol{\Mod}_g$. Let $\Pb\H(\ul{k})$ and $\Pb\Mm$ denote respectively the projectivizations of $\H(\ul{k})$ and $\Mm$ in $\Pb\Omega\ol{\Mod}_g$.

Let $\hat{\Phi}: \U_\G \ra \Pb\Omega\Mod_g$ denote the composition of $\widetilde{\Phi}$ and the projection from $\Omega\Mod_g$ minus the zero section to $\Pb\Omega\Mod_g$.

\begin{Lemma}\label{lm:stab:cyl:diag:dense}
For any $\alpha \in \Alp_\G$, $\hat{\Phi}(\U_{\G,\alpha})$ is  an open dense subset of $\Pb\Mm$.
\end{Lemma}
\begin{proof}
Let us consider the map $\allowbreak \S^1\times\U_\G \ra \H(\ul{k}), (\theta,\bar{x}) \mapsto e^{\imath\theta}\cdot\widetilde{\Phi}(\bar{x})$.
One can easily check that  this map sends $\S^1\times \U_{\G,\alpha}$ to an open subset of $\Mm$  which is $\GL^+(2,\R)$-invariant. By the definition of affine submanifolds, such a subset must be dense and of full measure in $\Mm$. Since $\hat{\Phi}(\U_{\G,\alpha})$ is precisely the projectivization of this subset in $\Pb\Omega\ol{\Mod}_g$, the lemma follows.
\end{proof}

\section{A finite family of local equations for $\Mm$}\label{sec:finite:subsp}
\subsection{Statement of a finiteness result}
We continue using the notation in Section~\ref{sec:cyl:dec:charts}.
For $j=1,\dots,k$, let $\ell_j$ be the circumference of $C_j$. Note that each $\ell_j$ is a linear function of $(|a_1|,\dots,|a_{m_0}|)$ ($|a_i|$ is the length of $a_i$),
whose coefficients are determined by the diagram $\G$.
Our goal now is to show the following
\begin{Theorem}\label{thm:good:subsp:finite:G}
Set
$$
\U^*_\G:=\{(r_1,\dots,r_{m_0},z_1,\dots,z_k) \in \U_\G \; | \; 0\leq \re(z_j) < \ell_j, \, j=1,\dots,k\} \text{ and } \U^*_{\G,\alpha}:=\U^*_\G\cap V_\alpha.
$$
Then the set $\Alp_\G^*=\{\alpha\in \Alp_\G, \; \U^*_{\G,\alpha} \neq \varnothing\}$ is finite.
\end{Theorem}
\begin{Remark}\label{rk:finiteness:from:MirWri}\hfill
\begin{itemize}
\item[-] If $M=(X,\omega)\in \Mm$ admits a stable cylinder decomposition in the horizontal direction, then we can always choose the crossing saddle connections $b_j$ such that $0 \leq \re(\omega(b_j)) < \ell_j$, which means that $M \in \widetilde{\Phi}(\U^*_\G)$ for some cylinder diagram $\G$. Theorem~\ref{thm:good:subsp:finite:G} means that the system of linear equations defining $\Mm$ in a neighborhood of $M$ belongs to a finite family.

\item[-] This result can be derived as a special case from the arguments of the proof of \cite[Th.5.1]{Mir-Wri:boundary}.\footnote{The author thanks A. Wright for pointing it out to him.} We will give here below an independent proof based on the fact that $\Mm$ is an algebraic subvariety of $\H(\ul{k})$ (see~\cite{Fil:algebraic}), and on the analogies with Teichm\"uller curves.
\end{itemize}
\end{Remark}
As  consequences of Theorem~\ref{thm:good:subsp:finite:G} we get the following (see also \cite[Th.1.4]{Mir-Wri:boundary}).
\begin{Corollary}\label{cor:ratio:finite}
 Let $M$ be horizontally periodic surface in $\Mm$ and $c$ a core curve of a horizontal cylinder on $M$. Let $c'$ be either a core curve of a horizontal cylinder, or a horizontal saddle connection joining a singularity of $M$ to itself. Then the ratio $|c'|/|c|$ belong to a finite set depending on $\Mm$.
\end{Corollary}
\begin{proof}
Assume that the cylinder decomposition of $M$ in the horizontal direction is $\Mm$-stable with diagram $\G$. Since we can always choose the crossing saddle connections $b_j$ such that $0\leq \re(\omega(b_j)) < \ell_j, \, j=1,\dots,k$, there exists $\alpha \in \Alp_\G^*$ such that $M$ is contained in the image of $V_\alpha\cap \U^*_\G$ by $\widetilde{\Phi}$.
Since the number of diagram for cylinder decompositions of surfaces in  $\H(\ul{k})$ is finite, and for each $\G$, the set $\Alp_\G^*$ is finite by theorem~\ref{thm:good:subsp:finite:G}, the corollary is proved for this case.

If $M$ is not $\Mm$-stably horizontally periodic, then we can deform $M$ by using some small vector in $T^{\imath\R}_M\Mm\cap\ker(\pp)$ (where $\pp: H^1(M,\Sig,\C) \ra H^1(M,\C)$ is the natural projection) to get a surface $\Mm$-stably horizontally periodic, on which both $c$ and $c'$ persist and remain horizontal. We can then conclude by the argument above.
\end{proof}


\subsection{``Cusps'' of rank one submanifolds}
For any $R>0$, set
$$
\U_\G(R)=\{(r_1,\dots,r_{m_0},x_1,y_1,\dots,x_k,y_k) \in \U_\G \, | \, y_j>R\ell_j, \, j=1,\dots,k\}.
$$
We also define for any $\alpha \in \Alp_\G$, $ \allowbreak \U_{\G,\alpha}(R):=V_\alpha\cap \U_\G(R)$. 
We will see that, for each $\alpha \in \Alp_\G$, if $R>1$, $\hat{\Phi}$ sends $\U_{\G,\alpha}(R)$ onto a ``cusp'' of $\ol{\Pb\Mm}$. Our strategy to prove Theorem~\ref{thm:good:subsp:finite:G} is first to show that the set of cusps of $\ol{\Pb\Mm}$ is finite, and then for any fixed cusp, the set of $\alpha \in \Alp_\G^*$ such that $\U_{\G,\alpha}(R)$ is mapped to this cusp is finite.

Let $M=\widetilde{\Phi}(\bar{x}) \in \H(\ul{k})$ with $\bar{x}\in \U_\G$. For $t\in \R$, let $a_t:= \left( \begin{smallmatrix} 1 & 0 \\ 0 & e^t \end{smallmatrix}\right)$ and $M_t:=a_t\cdot M$.
As $t \ra +\infty$, $M_t$ converges to  a stable differential $(X_\infty,\omega_\infty)$, where $X_\infty$ is a stable curve obtained from $M$ by pinching the core curves of all of its horizontal cylinders, and $\omega_\infty$ has a simple pole at every node of $X_\infty$.
Note that the topology of $X_\infty$ is completely determined by the cylinder diagram $\G$.
The stable differential  $(X_\infty,\omega_\infty)$ belongs a stratum $\Sc_\G$ of $\Omega\ol{\Mod}_g$  which is contained in the closure of $\H(\ul{k})$.
Let $[(X_\infty,\omega_\infty)]$ and $\Pb\Sc_\G$ denote the projectivizations of $(X_\infty, \omega_\infty)$ and $\Sc_\G$ in $\Pb\Omega\ol{\Mod}_g$.

\begin{Lemma}\label{lm:cusp:unique:intersection}
If $R>1$, then $[(X_\infty,\omega_\infty)]$ is the unique intersection of the closure of $\hat{\Phi}(\U_{\G,\alpha}(R))$ with $\Pb\Sc_\G$, that is
$\ol{\hat{\Phi}(\U_{\G,\alpha}(R))} \cap \Pb\Sc_\G=\{[(X_\infty,\omega_\infty)]\}$.
\end{Lemma}
\begin{proof}
Recall that $M$ is defined by an Abelian differential $(X,\omega)$.
We can suppose that  $\allowbreak \min\{\ell_1,\dots,\ell_k\}=1$.
Let $c_j$ be a core curve of the cylinder $C_j$ on $M$, $j=1,\dots,k$.
We have a map $f_\infty: X \ra X_\infty$ that satisfies
\begin{itemize}
  \item each zeros of $\omega$ is mapped to a zero of $\omega_\infty$ of the same order,
  \item $c_j$ is mapped to a node $\mathbf{n}_j$ of $X_\infty$,
  \item $f$ maps $X\setminus (\cup_{j=1}^k c_j)$ homeomorphically to $X_\infty\setminus (\cup_{j=1}^{k}\mathbf{n}_j)$.
 \end{itemize}
The flat surface $M_\infty$ defined by $(X_\infty,\omega_\infty)$ can be constructed as follows:  for $j=1,\dots,k$, cut the cylinder $C_j$ along   $c_j$,  then on the  resulting surface, glue to each boundary component arising from $c_j$  a half-infinite horizontal cylinder with  the same circumference as $C_j$.

 Let $\mathbf{a}$ denote the union of all the horizontal saddle connections of $M$.  By construction those saddle connections persist on $M_\infty$.
 For any  positive real number  $h >0$, let $U_h$ denote the subset of $M_\infty$ consisting of points whose distance to $\mathbf{a}$ is at most $h$. Note that the complement of $U_h$ is the disjoint union of $2k$ half-infinite cylinders.
 Clearly, we have $\cup_{h>0}U_h=X_\infty\setminus\{\cup_{j=1}^k\mathbf{n}_j\}$.

 Let us recall a notion of convergence compatible with the topology of $\Omega\ol{\Mod}_g$ in terms of flat metric (see {\em e.g.} \cite[Sec. 5.3]{Bainbridge-GT}, or
 \cite[Def. 2.2]{Mir-Wri:boundary}).
 A sequence $\{(X_i,\omega_i)\}_{i\in \N}$ of Abelian differentials in $\H(\ul{k})$ converges to $(X_\infty,\omega_\infty)$ if there exist a sequence of maps $f_i: X_i \ra X_\infty$ collapsing some curves of $X_i$ to the nodes of $X_\infty$ such that,  for $i$ large enough, the restriction of $f_i$ to $f_i^{-1}(U_h)$ is a diffeomorphism onto its image, and  the sequences of metrics $({f_i}_{|U_h}^{-1})^*\omega_i$ on $U_h$ converges to $\omega_\infty$ in the compact open topology.

Let $(X_i,\omega_i)=\widetilde{\Phi}(\bar{x}_i)$, with $\bar{x}_i \in \U_{\G,\alpha}(R)$, be a sequence such that $[(X_i,\omega_i)]$ converges to a point $[(X'_\infty,\omega'_\infty)]$ in $\Pb\Sc_\G$.
We can normalize $(X_i,\omega_i)$ such that the smallest circumference of the horizontal cylinders on $X_i$  equals to $1$, and assume that $(X_i,\omega_i)$ converges to $(X'_\infty,\omega'_\infty)$ in $\Omega\ol{\Mod}_g$. We will write
$$
\bar{x}_i=(r^{(i)}_1,\dots,r^{(i)}_{m_0},x_1^{(i)},y_1^{(i)},\dots, x_k^{(i)},y_k^{(i)}).
$$
Since the ratios $r_{i_1}/r_{i_2}$ are constant on $\U_{\G,\alpha}$, the condition that the smallest circumference of $C_1,\dots,C_k$  equals to $1$ implies that $(r^{(i)}_1,\dots,r^{(i)}_{m_0})$ does not depend on $i$.

Given $h>0$, we choose $R'$ such that $R'>\max\{2h,R\}$. Since $X'_\infty$ has $k$ nodes, for $i$ large enough  $(X_i,\omega_i)$ must have $k$ disjoint cylinders with moduli greater than $R'$. We claim that those cylinders are $C_1,\dots,C_k$. Assume that there is a cylinder $C$ not in this family with modulus $> R'$, then $C$ must cross one of them, say $C_1$. In this case the circumference $\ell(C)$ of $C$ is at least $y_1^{(i)}$, and the height $h(C)$ of $C$ cannot exceed $\ell_1$. Thus we would have
$$
R' < \frac{h(C)}{\ell(C)} \leq \frac{\ell_1}{y_1^{(i)}} < \frac{1}{R},
$$
which is impossible since we have $R'>R>1$.
It follows that we have $y_j^{(i)}/\ell_j > R', \, j=1,\dots,k$.

Let $X_{i,h}$ be the subset of $X_i$ consisting of points whose distance to the the union of the horizontal saddle connections is at most $h$. Since $y_j > 2h\ell_j, j=1,\dots,k$, $X_{i,h}$ is a proper subset of $X_i$ which is isometric to $U_h \subset X_\infty$.
We can then  define a map $f_i: X_i \ra X_\infty$ which collapses some  curves in $X_i\setminus X_{i,h}$ to the nodes of $X_\infty$, such that the restriction ${f_i}_{|X_{i,h}}: X_{i,h} \ra U_h$ is an isometry. It follows that we can extract from $\{(X_i,\omega_i)\}$ a subsequence converging to $(X_\infty,\omega_\infty)$. Thus we must have $(X'_\infty,\omega'_\infty)=(X_\infty,\omega_\infty)$, and the lemma follows.
\end{proof}

Let $\ol{\Pb\Mm}$ denote the closure (in the usual topology)  of $\Pb\Mm$ in $\Pb\Omega\ol{\Mod}_g$.
By a result of S.~Filip~\cite{Fil:algebraic},  $\ol{\Pb\Mm}$ is a subvariety of $\Pb\Omega\ol{\Mod}_g$ which contains $\Pb\Mm$ as a Zariski open subset.
The germ of complex analytic sets defined by $\ol{\Pb\Mm}$ at $[(X_\infty,\omega_\infty)]$ has finitely many irreducible components.
From general results on complex analytic sets,  each irreducible component $\Cc$ of this germ satisfies the following property: there is a basis of neighborhoods $\{\VV_i, i \in I\}$ of $[(X_\infty,\omega_\infty)]$ in $\Pb\Omega\ol{\Mod}_g$ such that for any $i \in I$,  $(\VV_i \cap\Cc)\cap\Pb\Mm$ is path connected (see~\cite[Prop. A.5, p.117]{Mumford76}).

\begin{Lemma}\label{lm:Ualpha:cusp}
Assume that $\bar{x}\in V_\alpha$, with some $\alpha \in \Alp_\G$. Then  there is an irreducible component $\Cc_\alpha$ of the germ of  $\ol{\Pb\Mm}$ at $[(X_\infty,\omega_\infty)]$ such that,  for any $R>1$, there exists a neighborhood $\VV$ of $[(X_\infty,\omega_\infty)]$ in $\Pb\Omega\ol{\Mod}_g$ such that $\hat{\Phi}(\U_{\G,\alpha}(R))$ contains the set $(\VV\cap\Cc_\alpha)\cap\Pb\Mm$. \end{Lemma}
\begin{Remark}\label{rk:cusp}
When $\Mm$ is a closed $\GL^+(2,\R)$-orbit, $\Cc_\alpha$ is a cusp of the corresponding Teichm\"uller curve.
\end{Remark}
\begin{proof}

Let $M_\infty=(X_\infty,\omega_\infty)$, and $\hat{M}_\infty=[(X_\infty,\omega_\infty)]$ be the projectivization of $M$ in $\Pb\Omega\ol{\Mod}_g$.
Let $\UU$ be a neighborhood of $\hat{M}_\infty$ in $\Pb\Omega\ol{\Mod}_g$ such that $\UU$ intersects the set $\ol{\Pb\Mm}_{\rm reg}$ of regular points of $\ol{\Pb\Mm}$ in a finite union of disjoint complex manifolds (of dimension $d-1$), each of which corresponds to an irreducible component of the germ of $\ol{\Pb\Mm}$ at $\hat{M}_\infty$.
The arguments of Lemma~\ref{lm:cusp:unique:intersection} actually show that if $R>0$ is large enough, then $\hat{\Phi}(\U_{\G,\alpha}(R))$ is contained in $\UU$.
Since $\U_{\G,\alpha}(R)$ is an open connected subset of $\U_{\G,\alpha}$, its image under $\hat{\Phi}$ is an open connected subset of $\Pb\Mm$ (Lemma~\ref{lm:stab:cyl:diag:dense}).
It follows that $\hat{\Phi}(\U_{\G,\alpha}(R))$ is contained in a unique irreducible germ $\Cc_\alpha$ of the analytic sets defined by $\ol{\Pb\Mm}$.

Let $\VV\subset \UU$ be a neighborhood of $\hat{M}_\infty$ in $\Pb\Omega\ol{\Mod}_g$ such that every translation surface that projects to a point in $\VV\cap\Pb\H(\ul{k})$ has $k$ disjoint cylinders (not necessarily parallel) each of which has modulus greater than $2R$ (such a neighborhood exists because $M_\infty$ has $k$ infinite cylinders).
We will show that $(\VV\cap\Cc_\alpha)\cap \Pb\Mm$ is contained in $\hat{\Phi}(\U_{\G,\alpha}(R))$.

Let $\hat{M}$ be projectivization  of $M$ in $\Pb\Mm$. Without loss of generality, we can assume that $\hat{M} \in (\VV\cap\Cc_\alpha)\cap\Pb\Mm$.
Since the complement of $(\VV\cap\Cc_\alpha)\cap\Pb\Mm$  in $\VV\cap\Cc_\alpha$ is a proper analytic subset of  $\VV\cap\Cc_\alpha$, the set $(\VV\cap\Cc_\alpha)\cap\Pb\Mm$ is connected.
Let $N$ be a surface in $\Mm$ whose projectivization $\hat{N}$ belongs to $(\VV\cap\Cc_\alpha)\cap\Pb\Mm$.
Then there is a path $\gamma: [0,1] \ra (\VV\cap\Cc_\alpha)\cap\Pb\Mm$ from $\hat{M}$ to $\hat{N}$.
By a slight abuse of notation, let us denote also by $\gamma$ a lift of this path in $\Mm$ joining  $M$ and $N$.

We claim that there exists a path $\widetilde{\gamma}: [0,1] \ra \S^1\times \U_{\G,\alpha}(R)$, such that $\widetilde{\Phi}\circ\widetilde{\gamma}=\gamma$ (here by a slight abuse of notation, we denote by $\widetilde{\Phi}$ the map $\S^1\times\U^\alpha_\G  \ra \Mm, \;  (\theta,x) \mapsto e^{\imath\theta}\cdot\widetilde{\Phi}(x)$).
Since the map $\widetilde{\Phi}: \S^1\times\U_\G^\alpha(R) \ra \Mm$ is locally homeomorphic, there is a maximal $s_0 \in (0,1]$ such that one can define a path $\widetilde{\gamma}:[0,s_0)\ra \S^1\times \U_{\G,\alpha}(R)$ satisfying $\widetilde{\Phi}\circ\widetilde{\gamma}(s)=\gamma(s)$, for all $s\in  [0,s_0)$.
We will write $\widetilde{\gamma}(s)=(\theta(s),\bar{x}(s))$, with $\bar{x}(s)\in \U_{\G,\alpha}(R)$.

Let $N_s:=\gamma(s) \in \Mm$, for $s\in [0,1]$ (recall that $\gamma(s)$ is defined for all $s \in [0;1]$). By the definition of affine manifolds,  in a neighborhood of $N_{s_0}$ in $\H(\ul{k})$, $\Mm$ is identified with a linear subspace $V$ of $H^1(M,\Sig,\C)$ (here we used a homeomorphism between $M$ and $N_{s_0}$ respecting the singularities). Thus, for $s< s_0$ close enough to $s_0$, we can suppose that $N_s \in V$. Note that $V^\R$ and $V_\alpha^\R$ are related by a basis change of $H_1(M,\Sig,\Z)$. Hence, as $s \ra s_0^-$, $\widetilde{\gamma}(s)$ converges to a point $(\theta(s_0),\bar{x}(s_0)) \in \S^1\times V_\alpha$. We need to show that $\bar{x}(s_0)\in \U_{\G,\alpha}(R)$. Since $\bar{x}(s)\in V_\alpha$, it follows that $\bar{x}(s_0)\in V_\alpha$. It remains to check that $\bar{x}(s_0)\in \U_\G(R)$.

For $s \in [0,s_0)$,  $N_{s}$ has $k$  cylinders corresponding  to horizontal cylinders $C_1,\dots,C_k$ of $M=N_0$. Since the ratio of the lengths of  any pair of saddle connections in the family $\{a_1,\dots,a_m\}$ is constant in $V_\alpha$, we can assume that the lengths of those saddle connections in $N_s$ are constant functions of $s$. Thus, we can write
$$
\bar{x}(s)=(r_1,\dots,r_{m_0},x_1(s),y_1(s),\dots,x_k(s),y_k(s)).
$$
It follows that  the circumferences of the cylinders $C_1,\dots,C_k$ are also constant functions of $s$.
Since $\bar{x}(s) \in \U_\G(R)$, we have $y_j(s)>R\ell_j$, for all $j=1,\dots,k$. By the assumption that $\hat{N}_s\in \VV\cap\Pb\Mm$, there are $k$ cylinders in $N_s$ with moduli at least $2R$. We claim that those cylinders must be $\{C_1,\dots,C_k\}$. Indeed, let $C$ be a cylinder in $N_s$ with modulus $\geq 2R$. If $C$ is not one of $C_1,\dots,C_k$ then $C$ must cross one of them, say $C_1$. Therefore the circumference $\ell(C)$ of $C$ satisfies $\ell(C) \geq y_1(s)$.
On the other hand, the height $h(C)$ of $C$ cannot exceed $\ell_1$. Thus we have
$$
2R\leq  \frac{h(C)}{\ell(C)} \leq  \frac{\ell_1}{y_1(s)} < \frac{1}{R},
$$
which is a contradiction since we have chosen $R>1$. It follows in particular that the modulus of the cylinder $C_j$ on $N_s$ is at least $2R$, which means that $y_j(s) \geq 2R\ell_j, \, j=1,\dots,k$.  As a consequence, since $y_j(s)$ converges to $y_j(s_0)$ as $s \ra s_0^-$, we draw that $y_j(s_0) \geq 2R\ell_j$. Therefore $\bar{x}(s_0) \in \U_{\G,\alpha}(R)$,
and the path $\widetilde{\gamma}$ can be extended to $[0,s_0]$. By compactness the lemma follows.
\end{proof}

\medskip

Let us now define $\Cb_\G$ to be the set of points $\hat{M}_\infty \in \Pb\Sc_\G$ for which there exists $\alpha \in \Alp_\G$ such that $M_\infty$ is contained in the closure of $\hat{\Phi}(\U_{\G,\alpha}(R))$, for some $R>1$.

\begin{Lemma}\label{lm:cusp:finite}
 The set $\Cb_\G$ is finite.
\end{Lemma}
\begin{proof}
Let $\ol{\Pb\Sc}_\G$ be the closure of the stratum $\Pb\Sc_\G$ in $\Pb\Omega\ol{\Mod}_g$. Note that $\ol{\Pb\Sc}_\G$ is a subvariety of $\ol{\Pb\H(\ul{k})}$.
Assume that $\Cb_\G$ is an infinite subset of $\Pb\Sc_\G$. Then by compactness, there exists a sequence $\{\hat{M}^{(i)}_\infty\}_{i \in \N} \subset \Cb_\G$ that converges to a point $\hat{M}_\infty \in \ol{\Pb\Sc}_\G$.

Since $\hat{M}^{(i)}_\infty$ has $k$ infinite cylinders,  the flat surface $\hat{M}_\infty$ must have $k'\geq k$ infinite cylinders.
We first claim that $k'=k$. To see this, let $\VV$ be a neighborhood of $\hat{M}_\infty$ in $\ol{\Pb\H(\ul{k})}$ such that every flat surface represented by a point in $\VV$ has $k'$ disjoint cylinders with moduli greater than $2R$, where $R>1$.
Consider a point $\hat{M}^{(i)}_\infty \in \Cb_\G\cap \VV$. By assumption, $\hat{M}^{(i)}_\infty$ is the limit of $\hat{M}_t$  as $t\ra +\infty$,
where $\hat{M}_t$ is the projectivization of $a_t\cdot \widetilde{\Phi}(\bar{x})$ with $\bar{x} \in \U_{\G,\alpha}$, for some $\alpha \in \Alp_\G$.
Thus $\VV$ contains some flat surfaces that admit a cylinder decomposition with diagram $\G$ such that the modulus of any cylinder in this family is at least $R>1$.
Any other cylinder on those surfaces must cross some cylinders in this family. By the argument of Lemma~\ref{lm:Ualpha:cusp}, the modulus of such a cylinder is smaller than $1/R < 2R$.
Hence, there cannot exist more than $k$ disjoint cylinders of moduli greater than $2R$ on those surfaces, from which we deduce that $k'=k$.
It follows in particular that if $\bar{x}\in \U_\G(R)$ and the projectivization of $\widetilde{\Phi}(\bar{x})$ belongs to $\VV$, then $\bar{x}\in \U_\G(2R)$.

Let $\Cc_i$ be the irreducible component of $\VV\cap \ol{\Pb\Mm}$ that contains the germ of $\ol{\Pb\Mm}$ at $\hat{M}^{(i)}_{\infty}$ defined in Lemma~\ref{lm:Ualpha:cusp}.
We claim that $\Cc_i\cap\Pb\Sc_\G=\{\hat{M}^{(i)}_{\infty}\}$.  To see this, we first notice that  $\Cc_i\cap\Pb\Mm$ is an open dense connected subset of $\Cc_i$, since its complement is contained in an analytic proper subset of $\ol{\Pb\Mm}$.
By definition, $\Cc_i\cap\Pb\Mm$ contains the projectivization $\hat{M}$ of a surface $M=\widetilde{\Phi}(\bar{x})$, with $\bar{x}\in \U_{\G,\alpha}(2R)$. The arguments of Lemma~\ref{lm:Ualpha:cusp} show that $\Cc_i\cap\Pb\Mm$ is contained in $\hat{\Phi}(\U_{\G,\alpha}(R))$. But by Lemma~\ref{lm:cusp:unique:intersection}, we know that $\hat{M}^{(i)}_{\infty}$ is the unique intersection of $\Pb\Sc_\G$ and the closure of $\hat{\Phi}(\U_{\G,\alpha}(R))$.
Since any point in the intersection $\Cc_i\cap\Pb\Sc_\G$ must be a limit point of $\hat{\Phi}(\U_{\G,\alpha}(R))$, we conclude that $\Cc_i\cap\Pb\Sc_\G=\{\hat{M}^{(i)}_{\infty}\}$.

The above claim implies that $\Cc_i\neq \Cc_j$ if $i\neq j$, which means that $\VV\cap\ol{\Pb\Mm}$ contains infinitely irreducible components. But this is a contradiction, because $\ol{\Pb\Mm}$ is a subvariety of $\ol{\Pb\H(\ul{k})}$, and the lemma follows.
\end{proof}

We now define an equivalence relation $\sim$ on the set $\Alp_\G$ as follows
\begin{Definition}\label{def:equiv:alpha}
 Two elements $\alpha,\alpha' \in \Alp_\G$ are said to  be {\em equivalent} if
\begin{itemize}
\item[-] there exist $\bar{x}\in \U_{\G,\alpha}$, $\bar{x}'\in \U_{\G,\alpha'}$ such that $a_t\cdot \widetilde{\Phi}(\bar{x})$ and $a_t\cdot \widetilde{\Phi}(\bar{x}')$ converge to the same point $M_\infty$ in $\Sc_\G$ as  $t \ra +\infty$, and
\item[-] for $R>0$ large enough $\hat{\Phi}(\U_{\G,\alpha}(R))$ and $\hat{\Phi}(\U_{\G,\alpha'}(R))$ are contained in the same irreducible component of the germ of $\ol{\Pb\Mm}$ at  $\hat{M}_\infty$.
\end{itemize}
\end{Definition}

As a direct consequence of Lemmas~\ref{lm:cusp:finite}, we get
\begin{Corollary}\label{cor:finite:equiv:cls}
The relation $\sim$ on $\Alp_\G$ has finitely many equivalence classes.
\end{Corollary}
\begin{proof}
By definition, each equivalence class of this relation corresponds to an irreducible component  of the germ of analytic sets  of $\ol{\Pb\Mm}$ at a point  $M_\infty \in \Cb_\G$. Since $\Cb_\G$ is finite, and the germ of $\ol{\Pb\Mm}$ at each point of $\Cb_\G$ can only have finitely many irreducible components, the corollary follows.
\end{proof}

\begin{Lemma}\label{lm:Phi:finite:preim}
Let
$$
\bar{x}=(r_1,\dots,r_{m_0},x_1,y_1,\dots,x_k,y_k) \text{ and } \bar{x}'=(r'_1,\dots,r'_{m_0},x'_1,y'_1,\dots,x'_k,y'_k)
$$
be two points in $\U_\G(R)$, with $R>1$.
Assume that there exist $\theta$ such that  $\widetilde{\Phi}(\bar{x})=e^{\imath\theta}\cdot \widetilde{\Phi}(\bar{x}')$.  Then we must have  $\theta\in \Z\pi$, and  $\widetilde{\Phi}(\bar{x}')$ can be obtained from $\widetilde{\Phi}(\bar{x})$ by permuting the horizontal cylinders and by Dehn twists in the cylinders. In particular, there are two permutations $\sigma_1 \in \mathfrak{S}_{m_0}$ and $\sigma_2\in \mathfrak{S}_k$ such that $r'_i=r_{\sigma_1(i)}$ and $y'_j=y_{\sigma_2(j)}$.
\end{Lemma}
\begin{proof}
Let $M= \widetilde{\Phi}(\bar{x}), M'=\widetilde{\Phi}(\bar{x}')$.
 Assume that we have  $M=M'$, which  means that there is an isometry between $M$ and $M'$ respecting the holonomy of any saddle connection. This implies that a horizontal cylinder (resp. saddle connection) is mapped to a horizontal cylinder (resp. saddle connection). Therefore, the surface $M'$  can be obtained from  $M$ by a permutation of the horizontal cylinders and saddle connections, and by some Dehn twists in the cylinders. The last assertion follows immediately from this observation

Assume now that we have $M=e^{\imath\theta}\cdot M'$. This assumption means that $M$ admits a cylinder decomposition in the direction $\theta$ with the corresponding parameters given by $\bar{x}'$. Since $\bar{x}' \in \U_\G(R)$, all the moduli of the cylinders in the direction $\theta$ is greater than $R$. If $\theta \not\in \Z\pi$, the core curves of one of those cylinders must cross $C_1$, hence we can use the same argument as in Lemma~\ref{lm:Ualpha:cusp} to get a contradiction. Therefore, we must have $\theta \in \Z\pi$, and the lemma follows from the
argument above.
\end{proof}

\begin{Lemma}\label{lm:coord:change:equiv:subsp}
Assume that $\alpha \sim {\alpha'}$ in $\Alp_\G$. Then there exist two permutations $\sigma_1\in \mathfrak{S}_{m_0}$ and $\sigma_2\in \mathfrak{S}_k$ such that if $\allowbreak \bar{x}=(r_1,\dots,r_{m_0},x_1,y_1,\dots,x_k,y_k)$ is a point  in $\U_{\G,\alpha}$, then $\U_{\G,\alpha'}$ contains a point $\allowbreak \bar{x}'=(r'_1,\dots,r'_{m_0}, x'_1,y'_1,\dots,x'_k,y'_k)$ where $r'_i=r_{\sigma_1(i)}$ and $y'_j=y_{\sigma_2(j)}$, for $i=1,\dots,m_0,$ and $ j=1,\dots,k$.
\end{Lemma}
\begin{proof}
Let $M=\widetilde{\Phi}(\bar{x})$ and $M_\infty \in \Sc_\G$ be the limit of $a_t\cdot M$ as $t \ra +\infty$.  Let $\Cc$ be  the  irreducible component of the germ of  $\ol{\Pb\Mm}$ at $\hat{M}_\infty$ defined  in Lemma~\ref{lm:Ualpha:cusp}.
We can assume that $\hat{\Phi}(\U_{\G,\alpha'}(R))$ contains a subset  of the form $(\VV\cap\Cc)\cap \Pb\Mm$, where $\VV$ is a neighborhood of $\hat{M}_\infty$ in $\ol{\Pb\H(\ul{k})}$.

By assumption, when $t$ is large enough, the projectivization $\hat{M}_t$ of $M_t=a_t\cdot M$ belongs to $\VV\cap\Cc$. Thus, there exists $\allowbreak \bar{x}' \in \U_{\G,\alpha'}(R)$, such that $\hat{M}_t=\hat{M}'$, where $\hat{M}'$ is the projectivization of $M':=\widetilde{\Phi}(\bar{x}')$.
We can normalize $M'$ such that $\Aa(M_t)=\Aa(M')$, which means that $M'=e^{\imath\theta}\cdot M_t$, for some $\theta\in \R$.

If we write $\allowbreak \bar{x}'=(r'_1,\dots,r'_{m_0},x'_1,y'_1,\dots,x'_k,y'_k)$, then by Lemma~\ref{lm:Phi:finite:preim}, we know that there exist two permutations $\sigma_1\in \mathfrak{S}_{m_0}$ and $\sigma_2\in \mathfrak{S}_k$ such that $r'_i=r_{\sigma_1(i)}$, $y'_j=e^ty_{\sigma_2(j)}$. Applying $a_{-t}$ to $M'$ we get the desired conclusion.
\end{proof}

\subsection{Proof of Theorem~\ref{thm:good:subsp:finite:G}}\label{sec:prf:thm:finite:subsp}
\begin{proof}
We will show that for any  $\alpha_0 \in \Alp_\G$, there are only finitely many $\alpha\in \Alp_\G$ equivalent to $\alpha_0$ such that $\U^*_{\G,\alpha}\neq \varnothing$. Since we have shown that the number of equivalence classes in $\Alp_\G$ is finite, this is enough to prove the theorem.

Recall that for every $\alpha\in \Alp_\G$, we can write $V_\alpha=V^\R_\alpha\oplus \imath W^\R_\alpha$, where $V_\alpha^\R$  is isomorphic to $V^\R=T^\R_{\widetilde{\Phi}(\bar{x})}\Mm$, $W_\alpha^\R$ is isomorphic to   a hyperplane $W^\R$ in $V^\R$ via the linear map $d\widetilde{\Phi}(\bar{x})$ with $\bar{x}$ being any point in $\U_{\G,\alpha}$. Let
$$
V^\R_+:=\{v \in V^\R, \; \langle v, c_1 \rangle >0\} \text{ and } W^\R_+:=\{v \in V^\R, \; \langle v, c_1\rangle =0, \langle v, b_j \rangle >0, \, j=1,\dots,k\}.
$$
Note that $V^\R_+$ and $W^\R_+$ are convex open cones in $V^\R$ and $W^\R$ respectively.
We define $V^\R_{\alpha+}$, and $W^\R_{\alpha+}$ to be the pre-images of $V^\R_+$ and $W^\R_+$ by $d\widetilde{\Phi}(\bar{x})$ respectively.
We then have the following identification
$$
\U_{\G,\alpha} \simeq V^\R_{\alpha+}\times W^\R_{\alpha+}.
$$
Fix a point $\bar{x}^0=(r^0_1,\dots,r^0_{m_0},x^0_1,y^0_1,\dots,x^0_k,y^0_k)$ in $\U_{\G,\alpha_0}$. Let $\alpha\in \Alp_\G$ be an element equivalent to $\alpha_0$. By Lemma~\ref{lm:coord:change:equiv:subsp}, we know that $\U_{\G,\alpha}$ contains a point $\allowbreak \bar{x}=(r_1,\dots,r_{m_0},x_1,y_1,\dots,x_k,y_k)$ such  that $(r_1,\dots,r_{m_0})=(r^0_1,\dots,r^0_{m_0})$ and $(y_1,\dots,y_k)=(y^0_1,\dots,y^0_k)$ up to some permutations of the indices.
By assumption, $\U_{\G,\alpha}$ contains a point $\allowbreak \bar{x}'=(r'_1,\dots,r'_{m_0},x'_1,y'_1,\dots,x'_k,y'_k) \in \U^*_{\G,\alpha}$, which means that
$$
0\leq x'_j < \ell'_j, \, j=1,\dots,k \qquad (*)
$$
where $\ell'_j$ is the circumference of the cylinder $C_j$, which is a linear function of $(r'_1,\dots,r'_{m_0})$.
By  Lemma~\ref{lm:r:proportional}, we have $\allowbreak (r_1,\dots,r_{m_0})= \lbd (r'_1,\dots,r'_{m_0})$, for some $\lbd \in \R_{>0}$.
By replacing $\allowbreak (r'_1,\dots,r'_{m_0},x'_1,\dots,x'_k)$ by $\allowbreak \lbd(r'_1,\dots,r'_{m_0},x'_1,\dots,x'_k)$, which also satisfies $(*)$,
we can assume that $(r'_1,\dots,r'_{m_0})=(r^0_1,\dots,r^0_{m_0})$ up to some permutation of indices.

Since $\allowbreak (r'_1,\dots,r'_{m_0},x'_1,\dots,x'_k)\in V^\R_{\alpha+}$ and $\allowbreak (y_1,\dots,y_k)\in W^\R_{\alpha+}$, from the identification $\U_{\G,\alpha}\simeq V^\R_{\alpha+}\times W^\R_{\alpha+}$,  we see that $\U_{\G,\alpha}$ contains the point $\allowbreak \bar{x}=(r'_1,\dots,r'_{m_0}, x'_1,y_1,\dots,x'_k,y_k)$. As $\allowbreak (r'_1,\dots,r'_{m_0},x'_1,\dots,x'_k)$  satisfies $(*)$, we get $\bar{x} \in \U^*_{\G,\alpha}$.

Now, assume  that there are infinitely many $\alpha_q \in \Alp_\G, \, q\in \N$, that are equivalent to $\alpha_0$ such that $\U^*_{\G,\alpha_q}\neq \varnothing$. Up to some permutation of indices, we get a sequence of points $\{\bar{x}_q\}$, where $\bar{x_q} \in  \U^*_{\G,\alpha_q}$ and $\allowbreak \bar{x}_q=(r^0_1,\dots, r^0_{m_0},   x^{(q)}_1,y^0_1, \dots ,x^{(q)}_k,y^0_k)$.
Since $(r^0_1,\dots,r^0_{m_0})$ and $(y^0_1,\dots,y^0_k)$ are fixed, the condition $(*)$ means that $\{\bar{x}_q\}$ is contained in a compact subset of $\U_\G$. Let $\bar{x}_\infty$ be an accumulation point of the sequence $\{\bar{x}_q\}$.
Using the fact that $\widetilde{\Phi}$ is a locally injective map, we draw that  $\Mm$ intersects  any neighborhood of $\widetilde{\Phi}(\bar{x}_\infty)$ in infinitely many irreducible components, which is impossible since $\Mm$ is a submanifold of $\H(\ul{k})$. The theorem is then proved.
\end{proof}


\section{Integration on local charts by cylinder decomposition}\label{sec:int:cyl:charts}
\subsection{Small saddle connections in a stable cylinder decomposition}
Before getting to the proof of Theorem~\ref{thm:estimate:eps:rk1}, we collect some facts about surfaces in a rank one affine submanifold with several small saddle connections.
In what follows $M$ is a $\Mm$-stably horizontally  periodic surface with cylinder diagram $\G$. We will use the same notation as in Section~\ref{sec:cyl:dec:charts}. Recall that the crossing saddle connections $b_j$ satisfy the condition $0\leq \re(\omega(b_j)) < \ell_j, \, j=1,\dots,k$.

By a {\em  cycle of horizontal cylinders} of $M$ we will mean a family $\Cc=\{C_1,\dots,C_j\}$ of horizontal cylinders such that the top of $C_i$ and the bottom of $C_{i+1}$ have a common  horizontal saddle connection, for $i\in \{1,\dots,j\}$ with the convention $C_{j+1}=C_1$.
Let  $h_i$ be the height of $C_i$. We define the total height of $\Cc$ to be $h_1+\dots+h_j$, and denote it by $h(\Cc)$.

For any direction $\theta \in \S^1$, let us denote by $\Sys(M,\theta)$ the minimal length of the  saddle connections  in direction  $\theta$ of $M$.
By convention, $\Sys(M,\theta)=+\infty$ if there is no saddle connection in direction $\theta$.

In the following lemma, we  show that if the area is fixed and the circumferences of the horizontal is bounded above then a cycle of cylinders cannot collapse simultaneously.

\begin{Lemma}\label{lm:cyl:non:collapse}
 For any positive real constant $\kappa >0$, there is a constant $\tilde{\eps}_0=\tilde{\eps}_0(\kappa)$ such that if $\Sys(M,0) \leq \kappa\sqrt{\Aa(M)}$,
 then $h(\Cc) >\tilde{\eps}_0\sqrt{\Aa(M)}$ for any cycle of horizontal cylinders on $M$.
\end{Lemma}
\begin{proof}
Assume that there is a sequence of $\Mm$-stably horizontally periodic surfaces $\{M_q\}_{q\in \N}\subset \Mm$, with $\Aa(M_q)=1$, $\Sys(M_q,0) \leq \kappa$, and $M_q$ contains a cycle of horizontal cylinder $\Cc_q$ such that $\lim_{q \ra +\infty} h(\Cc_q)=0$. Actually, we can assume that $\Sys(M_q,0)=\kappa$, for all $q \in \N$, since if we have $\Sys(M_q,0) =\varrho < \kappa$, we can replace  $M_q$ by $\left(\begin{smallmatrix}  e^t & 0 \\ 0 & e^{-t} \end{smallmatrix}\right)\cdot M_q$, with $t=\log(\kappa/\varrho)$.

We can assume that the diagrams of the cylinder decomposition in the horizontal direction of $M_q$ are the same for all $q \in \N$. Using a model surface $M$, we can number the horizontal cylinders of $M_q$ by $C_1,\dots,C_k$ in a consistent way. Furthermore, we can assume that $\Cc_q=\{C_1,\dots,C_j\}$ for all $q\in \N$.

By the definition of cycle of cylinders, there exists a closed curve $c$ on $M$ that crosses each core curve of $C_i$ once,  $i=1,\dots,j$. Let $(x_q,y_q)$ be the period of $c$ in $M_q$. We  have $y_q=h(\Cc_q)>0$, and $\lim_{q \ra +\infty} y_q=0$ by assumption.

By Theorem~\ref{thm:good:subsp:finite:G}, we can assume that  $M_q \in \widetilde{\Phi}(\U^*_{\G,\alpha})$ with some fixed $\alpha \in \Alp_\G^*$ for all $q \in \N$. Consequently, the condition on the lengths of the horizontal saddle connections implies that the circumferences of the horizontal cylinders  on the family $M_q$ are bounded above uniformly.
Therefore, we have $\lim_{q \ra +\infty} (\Aa(C_1)+\dots+\Aa(C_j))=0$, which means that the cycle $\Cc_q$ cannot contain all the horizontal cylinders, that is $j<k$.

Consider now the action of the horocycle flow $u_t:=\left(\begin{smallmatrix} 1 & t \\ 0 & 1 \end{smallmatrix}\right)$ on $M_q$. Since $\Sys(M_q,0)=\kappa$ for all $q\in \N$, by a fundamental result of Minsky-Weiss\cite{MinWei02},  the horocycle orbit of $M_q$ intersects a fixed compact subset of $\H(\ul{k})$. Thus by replacing $M_q$ by some point in $\{u_t\cdot M_q, \, t\in \R\}$ (which does not change the heights of the horizontal cylinders), we can assume that $M_q$ converges to a point $M_\infty \in \H(\ul{k})$ as $q \ra +\infty$.

Since $\lim_{q \ra +\infty} h(C_i)=0$, for $i=1,\dots,j$, the family of cylinders $\{C_1,\dots, C_j\}$ collapse to a union of horizontal saddle connections on $M_\infty$.
It follows that the period of $c$ in $M_\infty$ is $(x,0)$.
By considering the area, we see that some horizontal cylinders on $M_q$ must  remain on $M_\infty$. Thus, we can assume that $C_k$ remains on $M_\infty$.
Since $C_k$ is horizontal, the period of any core curve of $C_k$ and the period of $c$ are parallel as vectors in $\R^2$.

Note that we must have $M_\infty \in \Mm$. Since $\Mm$ is of rank one, a neighborhood of $M_\infty$ in $\Mm$ consists of surfaces $M'=A\cdot (M_\infty+v)$, where $A \in \GL^+(2,\R)$ close to the $\Id$, and $v$ is a small vector in $\ker(\pp)\cap T_{M_\infty}\Mm$. Since  $c$ is an element of $H_1(M,\Z)$ the period vectors of $c$ and of any core curve of $C_k$ remain parallel in a neighborhood of $M_\infty$ in $\Mm$. However, when $q$ is large enough, $M_q$ is contained in this neighborhood of $M_\infty$. But on $M_q$ those two vectors are not parallel since we have $y_q >0$. Thus we have a contradiction, which proves the lemma.
\end{proof}

\begin{Lemma}\label{lm:small:sc:finite:funct}
Assume that we have $\Sys(M,0) \leq \kappa \sqrt{\Aa(M)}$ for some $\kappa >0$.
Let $c$ be a non-horizontal saddle connection on $M$ such that $|c| <\Sys(M,0)$ and $|c| < \tilde{\eps}_0\sqrt{\Aa(M)}$, where $\tilde{\eps}_0$ is the constant of Lemma~\ref{lm:cyl:non:collapse}.   Then there is a unique ordered family $(j_1,\dots,j_s)$, where $j_\nu   \in \{1,\dots,k\}$ and $j_{\nu} \neq j_{\nu'}$ if $\nu \neq \nu'$, and a linear function $f$ of $(a_1,\dots,a_m)$ with integer coefficients such  that
$$
c=b_{j_1}+\dots+b_{j_s}+f(a_1,\dots,a_m) \in H_1(M,\Sig,\Z).
$$
Moreover, if the family $(j_1,\dots,j_s)$ is fixed, the function $f$ belongs to a finite family of cardinality at most $7^{s}m^{s+1}$.
\end{Lemma}
\begin{proof}
 We choose the orientation of $c$ to be upward.  Let $(x,y)$ be the period of $c$. We have $y>0$, and  $\allowbreak \max\{|x|,|y|\}  < \min\{\ell_1,\dots,\ell_k\}$ ($\ell_j$ is the circumference of $C_j$).
Observe that  $c$ crosses each horizontal cylinder at most once, since otherwise there would exist a cycle of horizontal cylinders whose total height is at most $ \tilde{\eps}_0\sqrt{\Aa(M)}$, which is a contradiction to Lemma~\ref{lm:cyl:non:collapse}.

Let $(C_{j_1},\dots,C_{j_s})$ be the sequence of horizontal cylinders crossed by $c$.
The starting point of $c$ must be the left endpoint of a the saddle connection $a_{i_0}$ in the bottom of $C_{j_1}$.
For $\nu=1,\dots,s-1$, let $a_{i_\nu}$ be the common saddle connection of the top of $C_{j_{\nu-1}}$ and the bottom of $C_{j_\nu}$ that intersects $c$.
Let $a_{i_s}$ be the saddle in the top of $C_{j_s}$ whose left endpoint is the terminating point of $c$.

For each $\nu \in \{1,\dots,s\}$, let $b'_{j_\nu}$ be the saddle connection in $C_{j_\nu}$ that joins the left endpoint of $a_{i_{\nu-1}}$ to the left endpoint of $a_{i_\nu}$ such that there is an embedded quadrilateral in $C_{j_{\nu}}$ bounded by $b'_{j_\nu}$, two horizontal segments $a'_{i_{\nu-1}}\subset a_{i_{\nu-1}}$ and $a'_{i_\nu} \subset a_{i_\nu}$, and a subsegment of $c$.
By construction, we  have
$$
c=b'_{j_1}+\dots+b'_{j_s} \in H_1(M,\Sig,\Z).
$$
Since $b_{j_\nu}-b'_{j_\nu}$ is a linear combination of the saddle connections in the top and bottom of $C_{j_\nu}$, we can write
$$
c=b_{j_1}+\dots+b_{j_s}+f(a_1,\dots,a_m) \in H_1(M,\Sig,\Z),
$$
where $f$ is a linear function with integer coefficients. It remains to show that $f$ belongs to a finite set once $(j_1,\dots,j_s)$ is fixed.

Recall that we have defined an equivalence relation on the set of crossing saddle connections of a cylinder horizontal cylinder $C$ as follows: two saddle connections are equivalent if they join the left endpoints of the same pair of saddle connections in the boundary of $C$.
To define $b_{j_\nu}$, we fix a pair of horizontal saddle connections, one in the top, the other in the bottom of $C_{j_\nu}$, and choose $b_{j_\nu}$ to be the unique saddle connection in the associated  equivalence class such that $0\leq \re(\omega(b_{j_\nu})) < \ell_{j_\nu}$.

We first observe that there exist  a fixed combination $f^0_{j_\nu}$ of horizontal saddle connections in the boundary of $C_{j_\nu}$, determined by the equivalence classes of $b_{j_\nu}$ and of $b'_{j_\nu}$,  and $n \in \Z$ such that
\begin{equation}\label{eq:crossing:rel}
b_{j_\nu}-b'_{j_\nu}= f^0_{j_\nu}+ nc_{j_\nu} \in H_1(M,\Sig,\Z),
\end{equation}
where  $c_{j_\nu}$ is a core curve of $C_{j_\nu}$.

Set $x_{j_\nu}:=\re(\omega(b_{j_\nu}))$ and $x'_{j_\nu}=\re(\omega(b'_{j_\nu}))$.
By definition,  there is an embedded  quadrilateral in $C_{j_\nu}$ that is formed by $b'_{j_\nu}$,  a subsegment $\check{c}$ of $c$, and two horizontal segments  $a'_{i_{\nu-1}}\subset a_{i_{\nu-1}}$ and $a'_{i_{\nu}} \subset a_{i_{\nu}}$.
Let $x'=\re(\omega(\check{c})), \; r'_{i_{\nu-1}}=\re(\omega(a'_{i_{\nu-1}})), \; r'_{i_\nu}= \re(\omega(a'_{i_\nu}))$.  Then
$$
|x'_{j_\nu}|=|r'_{i_{\nu-1}} +x'- r'_{i_\nu}| \leq |x'|+|r'_{i_{\nu-1}}-r'_{i_\nu}| \leq |x|+\max\{|a_{i_{\nu-1}}|, |a_{i_\nu}|\} < 2\ell_{j_\nu},
$$
and
$$
|x_{j_\nu}-x'_{j_\nu}| \leq |x_{j_\nu}|+ |x'_{j_\nu}| < 3\ell_{j_\nu}=3|\omega(c_{j_\nu})|.
$$
Relation \eqref{eq:crossing:rel} then implies
\begin{equation}\label{eq:condition:n}
|\omega(f^0_{j_\nu})+n \ell_{j_\nu}| < 3\ell{j_\nu}.
\end{equation}
Observe that  there are at most $7$ values of $n \in \Z$ such that the inequality \eqref{eq:condition:n} holds.
Since for each fixed sequence $(j_1,\dots,j_s)$,  we can have at most $m^{s+1}$ sequences $(a_{i_0}, \dots,a_{i_s})$, the lemma follows.
\end{proof}

\subsection{Integrations on the set of surfaces with several small saddle connections}
Having proved that one can cover a  full measure subset of $\Mm$ by a finite family of subsets of the form $\widetilde{\Phi}(\U^*_{\G,\alpha})$, we will now show that the integral of the function $e^{-\Aa}$ over the intersection of $\Mm(\eps^\nu)$ with each subset in this family is bounded by $O(\eps^{2\nu})$. Theorem~\ref{thm:estimate:eps:rk1} follows from this estimate.

Fix two positive real constants $\kappa$ and $\eps$.
Given a positive integer $\nu<d-1$, we define $\U^*_{\G,\alpha}(\kappa,\eps^\nu)$ to be the subset of $\U^*_{\G,\alpha}$ consisting of $\ol{x} \in \U^*_{\G,\alpha}$ such that the surface $M=\widetilde{\Phi}(\ol{x}) \in \Mm$ satisfies
\begin{itemize}
 \item[$\bullet$] $\Sys(M,0) < \kappa\sqrt{\Aa(M)}$,

 \item[$\bullet$] $M$ has $\nu$ non-horizontal saddle connections $e_1,\dots,e_\nu$ such that  $\allowbreak \max\{|e_1|,\dots,|e_\nu|\} < \eps\sqrt{\Aa(M)}$, and for any horizontal saddle connection $a$ of $M$, the family $\{a,e_1,\dots,e_\nu\}$ is linearly independent in $(T_M\Mm)^*$.
\end{itemize}
Given $M=(X,\omega) \in \widetilde{\Phi}(\U^*_{\G,\alpha})\subset \Mm$, pick  a horizontal saddle connection $a_i$ of $M$.
Let  $J:=\{j_1,\dots,j_{d-1}\}$ be a subset of $\{1,\dots,k\}$ such that $\{a_i,b_{j_1},\dots,b_{j_{d-1}}\}$ is basis of $(T_M\Mm)^*$.
 Then for any $j\in \{1,\dots,k\}$, $y_j$ is a linear functions of $(y_{j_1},\dots,y_{j_{d-1}})$ (where $y_j=\im(\omega(b_j))$).
Hence there is a  linear function $f$ of $(y_{j_1},\dots,y_{j_{d-1}})$, whose coefficients are determined by $(\G,\alpha)$ and the choice of $a_i$ and $J$, such that
$$
\Aa(M)=\sum_{j=1}^k \Aa(C_j)=\sum_{j=1}^k\ell_jy_j=|a_i|\cdot f(y_{j_1},\dots,y_{j_{d-1}}).
$$
Since there are finitely many choices for $a_i$ and $J$, we see that $f$ belongs to a finite set.
We define
$$
\chi(\G,\alpha):=\max \{||f|| \, : \, i \in \{1,\dots,m\}, J \subset \{1,\dots,k\}, \{a_i\}\cup \{b_j, \, j \in J\} \text{ is a basis of } (T_M\Mm)^*\}.
$$

\begin{Proposition}\label{prop:estimate:kap:eps:1}
There is a positive constant  $\tilde{K}$ depending on $(\G,\alpha)$, such that if $\eps < \frac{1}{\kappa\chi(\G,\alpha)}$, then
 $$
 \Ic:=\int_{\S^1\times\U^*_{\G,\alpha}(\kappa,\eps^\nu)} e^{-\Aa\circ\widetilde{\Phi}}\widetilde{\Phi}^*d\vol < \tilde{K} \kappa^{\nu+2}\eps^\nu.
 $$
\end{Proposition}
\begin{proof}
To simplify the notation,  we will omit the subscript $(\G,\alpha)$.
Consider a surface $M=\Phi(\bar{x})$ where $\bar{x}=(\bar{r},x_1,y_1,\dots,x_k,y_k) \in \U^*(\kappa,\eps^\nu)$.
We can assume that  $a_1$ is the smallest  horizontal saddle connection of $M$, that is  $|a_1|=\Sys(M,0)$.
By assumption, $M$ contains $\nu$ non-horizontal saddle connections $e_1,\dots,e_\nu$ of length at most $\eps\sqrt{\Aa(M)}$.
We can assume that the set of horizontal cylinders crossed by at least one of $e_1,\dots,e_\nu$ is $\{C_1,\dots,C_s\}$.
For $j\in \{1,\dots,s\}$, since the height of $C_j$ is bounded above by the length of some $e_i$, we must have  $0< y_j < \eps\sqrt{\Aa(M)}$.

Observe that the subspace of $(T_M\Mm)^*$ spanned by $\{a_1,b_1,\dots,b_s\}$ contains $e_1,\dots,e_\nu$.
Since the family $\{a_1,e_1,\dots,e_\nu\}$ is independent by assumption,  we draw that the dimension $d'$ of this subspace is at least $\nu+1$. Renumbering the cylinders if necessary, we can assume that $\{a_1,b_1,\dots,b_{d'}\}$ is a basis of this subspace. Note that since $d'\geq \nu$, we have $0 < y_j < \eps\sqrt{\Aa(M)}$ for $j\in \{1,\dots,\nu\}$.

We add to the family $\{a_1,b_1,\dots,b_{d'}\}$ some  saddle connections in $\{b_{s+1},\dots,b_k\}$ to get a basis $\Bc'$ of $(T_M\Mm)^*$.
We can assume that $\Bc'=\{a_1,b_1,\dots,b_{d-1}\}$. Using the basis $\Bc'$, we see that there is an $\R$-linear bijective map $\phi: \R\times \R^{2(d-1)} \ra V_\alpha$.

Let $r=|a_1|=\omega(a_1)$ and $x_j+\imath y_j=\omega(b_j), \, j=1,\dots,k$.
Note that the circumference of $C_j$ is given by $\ell_j=\lbd_jr$ where the constant  $\lbd_j$ is  determined by $(\G,\alpha)$,
Set
$$
\UU=\{(r,x_1,y_1,\dots,x_{d-1},y_{d-1}) \in \R_{>0}\times \R^{2(d-1)}, \; 0 \leq x_i < \lbd_i r, \; y_i > 0, \; i=1,\dots,d-1\}.
$$
By construction, we have $\phi^{-1}(\U^*_{\G,\alpha}) \subset \UU$.
Define
$$
\begin{array}{cccc}
\Phi: & \S^1\times\phi^{-1}(\U_{\G,\alpha}) &  \ra & \Mm \\
                                             & (\theta,r,x_1,y_1,\dots,x_{d-1},y_{d-1}) & \mapsto & e^{\imath\theta}\cdot \widetilde{\Phi}\circ\phi(r,x_1,y_1,\dots,x_{d-1},y_{d-1})
\end{array}
$$
Recall that the volume form $\vol$ on $\Mm$ is proportional to the Lebesgue measure. Thus there is a positive constant $\lbd$ such that
$$
\Phi^*d\vol=\lbd r d\theta drdx_1dy_1\dots dx_{d-1} dy_{d-1}.
$$
On $\UU$ the area function $\Aa\circ\Phi$ is given by $\Aa=r\cdot f$, where $f=\sum_{i=1}^{d-1}\xi_iy_i$ is a linear function of $(y_1,\dots,y_{d-1})$.
We claim that there exists $i>\nu$ such that $\xi_i \neq  0$. Indeed, if we have $\xi_{\nu+1}=\dots=\xi_{d-1}=0$,  then $f$ is a function of $(y_1,\dots,y_{\nu})$.
Since $\max\{|y_1|,\dots,|y_{\nu}|\} < \eps \sqrt{\Aa(M)}< \frac{\sqrt{\Aa(M)}}{\kappa\chi}$  and $r < \kappa\sqrt{\Aa(M)}$, we would have
$$
\Aa(M) < \kappa\sqrt{\Aa(M)}\cdot ||f||\cdot \eps \sqrt{\Aa(M)} \leq  \frac{\kappa\chi}{\kappa\chi}\Aa(M) = \Aa(M),
$$
which is a contradiction. Thus we can always assume that $\xi_{d-1} \neq 0$.
We will now use the following  change of variables on $\UU$, $(r,x_1,y_1,\dots,x_{d-1},y_{d-1}) \mapsto (r,x_1,y_1,\dots,x_{d-1},\Aa)$.
An elementary computation gives
$$
drdx_1dy_1\dots dx_{d-1}dy_{d-1}=\frac{1}{\xi_{d-1}r}drdx_1dy_1\dots dx_{d-1}d\Aa.
$$
Note that we also have
$$
\Aa(M)  \geq \sum_{i=\nu+1}^{d-2}\Aa(C_j)=r\cdot\sum_{i=\nu+1}^{d-2} \lbd_iy_i.
$$
Define
$$
\UU(\kappa,\eps^\nu):= \{(r,x_1,y_1,\dots,x_{d-1},y_{d-1}) \in \UU, \;  0< r <\kappa\sqrt{\Aa\circ\Phi}, \, 0< y_i < \eps\sqrt{\Aa\circ\Phi}, \; i=1,\dots,\nu\}.
$$
By assumption we have $\bar{x} \in \phi(\UU(\kappa,\eps^\nu))$. We have
\begin{eqnarray*}
\int_{\S^1\times\UU(\kappa,\eps^\nu)} e^{-\Aa\circ\Phi}\Phi^*d\vol & \leq & 2\pi \lbd \int_{\UU(\kappa,\eps^\nu)} e^{-\frac{\Aa}{2}-\frac{1}{2}r\sum_{j=\nu+1}^{d-2}\lbd_jy_j}
                                                                                                                            rdrdx_1dy_1\dots dx_{d-1}dy_{d-1} \\
                                                                   & \leq & \frac{2\pi\lbd}{\xi_{d-1}} \int_0^{+\infty}\int_{0}^{\kappa\sqrt{\Aa}}dr
                                                                   \left( \prod_{i=1}^{\nu} \int_0^{\lbd_ir}dx_i \int_0^{\eps\sqrt{\Aa}}dy_i\right)\times\\
                                                                   &      & \times \left(\prod_{i=\nu+1}^{d-2} \int_0^{\lbd_ir}dx_i\int_0^{+\infty} e^{-r\lbd_iy_i/2}dy_i\right)
                                                                   \left(\int_{0}^{\lbd_{d-1}r}dx_{d-1}\right) e^{-\Aa/2}d\Aa\\
                                                                   & \leq & \frac{2\pi\lbd\lbd_1\dots\lbd_{\nu}\lbd_{d-1}}{\xi_{d-1}}  2^{d-2-\nu} \eps^{\nu}\int_0^\infty  \left( \int_0^{\kappa\sqrt{\Aa}}r^{\nu+1} dr \right)\Aa^{\nu/2}e^{-\Aa/2}d\Aa\\
                                                                   & \leq & \tilde{K}\eps^{\nu}\kappa^{\nu+2}\\
\end{eqnarray*}
where  $\tilde{K}$ is a constant depending on $(\nu,\lbd,\lbd_1,\dots,\lbd_\nu,\lbd_d,\xi_{d-1})$. Since the basis $\Bc'$ belongs to a finite family, the set $\U^*(\kappa,\eps^\nu)$ is covered  finitely many sets  of the form $\phi(\UU^*(\kappa,\eps^\nu))$.
Thus the proposition follows from the inequality above.
\end{proof}

\begin{Corollary}\label{cor:estimate:eps:1}
There exist some constants $\eps_1>0$ and $K>0$ depending only on $\Mm$ such that if $\eps < \eps_1$, then
$$
\int_{\S^1\times\U^*_{\G,\alpha}(\eps,\eps^\nu)} e^{-\Aa\circ\widetilde{\Phi}}\widetilde{\Phi}^*d\vol < K\eps^{2(\nu+1)}.
$$
\end{Corollary}
\begin{proof}
The corollary follows from Proposition~\ref{prop:estimate:kap:eps:1} with  $\eps_1 >0$ such that $\eps_1 < \frac{1}{\eps_1\chi(\G,\alpha)}$, and $\kappa=\eps$.
\end{proof}

For each $(\G,\alpha)$ as above, we define $\U^*_{\G,\alpha}(\kappa,\eps^\nu)' \subset \U^*_{\G,\alpha}$ to be the set of $\bar{x} \in \U^*_{\G,\alpha}$ such that the surface
$M=\widetilde{\Phi}(\bar{x})$ satisfies $\Sys(M,0) < \kappa\sqrt{\Aa(M)}$, and $M$ has $\nu$ saddle connections $e_1,\dots,e_\nu$ such that
\begin{itemize}
\item[$\bullet$] $|e_i| < \eps\sqrt{\Aa(M)}, \; i=1,\dots,\nu$,

\item[$\bullet$] the family $\{e_1,\dots,e_\nu\}$ is independent in $(T_M\Mm)^*$.
\end{itemize}
Note that the difference between $\U^*_{\G,\alpha}(\kappa,\eps^\nu)$ and $\U^*_{\G,\alpha}(\kappa,\eps^\nu)'$ is that for the latter, we do not suppose that $\{a,e_1,\dots,e_\nu\}$ is an independent family in $(T_M\Mm)^*$ for any horizontal saddle connection $a$. In other words,  if $\bar{x}\in \U^*_{\G,\alpha}(\kappa,\eps^\nu)'$ then one of the $e_i$ could be a horizontal saddle connection.

\begin{Proposition}\label{prop:estimate:kap:eps:2}
For any fixed $\kappa >0$, there exist some positive real constants $\eps_0 >0$ and $K' >0$ such that if $0< \eps < \eps_0$ then
$$
\Ic':=\int_{\S^1\times\U^*_{\G,\alpha}(\kappa,\eps^\nu)'}e^{-\Aa\circ\widetilde{\Phi}}\widetilde{\Phi}^* d\vol < K'\eps^{2\nu}.
$$
\end{Proposition}
\begin{proof}
Since $\G$ and $\alpha$ is fixed, we will omit them in the notation through out the proof.
In what follows, we always suppose that $0< \eps < \kappa$. The existence of $\eps_0$ will be derived along the way.
Since $\eps < \kappa$, $\U^*(\kappa,\eps^\nu)'$ contains the set $\U^*(\eps,\eps^{\nu-1})$. Define
$$
\U^*(\kappa, \eps^\nu)'':=\U^*(\kappa,\eps^\nu)'\setminus\U^*(\eps,\eps^{\nu-1}).
$$
Since we have shown that  (see Corollary~\ref{cor:estimate:eps:1})
$$
\Ic:=\int_{\S^1\times\U^*(\eps,\eps^{\nu-1})}e^{-\Aa\circ\widetilde{\Phi}}\widetilde{\Phi}^*d\vol  < K\eps^{2\nu}
$$
it remains to show that
\begin{equation*}\label{eq:int:kap:esp:2}
\Ic'':=\int_{\S^1\times\U^*(\kappa,\eps^{\nu})''}e^{-\Aa\circ\widetilde{\Phi}}\widetilde{\Phi}^*d\vol  < K''\eps^{2\nu}.
\end{equation*}
for some constant $K''>0$ if $\eps>0$ is small enough.

Let $\tilde{\eps}_0>0$ be the constant of Lemma~\ref{lm:cyl:non:collapse}.
Consider a surface  $M=(X,\omega)=\widetilde{\Phi}(\bar{x})$ with $\bar{x}\in \U^*(\kappa,\eps^\nu)''$ and $\eps<\tilde{\eps}_0$. By assumption, we have
\begin{itemize}
\item[$\bullet$] $\eps\sqrt{\Aa(M)}\leq \Sys(M,0) < \kappa\sqrt{\Aa(M)}$, and

\item[$\bullet$] $M$ contains $\nu$ saddle connections $e_1,\dots,e_\nu$, none of which is horizontal, satisfying $|e_1| < \eps\sqrt{\Aa(M)}$ and $\{e_1,\dots,e_\nu\}$ is independent in $(T_M\Mm)^*$.
\end{itemize}
Since $|e_i| < \eps\sqrt{\Aa(M)} \leq  \Sys(M,0)$, by Lemma~\ref{lm:small:sc:finite:funct}, there is a finite set of linear combinations of $\{a_1,\dots,a_m,b_1,\dots,b_k\}$ such that the homology class of each $e_i$  in $H_1(M,\Sig,\Z)$ can be expressed by a function in this family.
In what follows, we assume that the horizontal $a_1$ of $M$ satisfies $|a_1|=\Sys(M,0)$, and choose for each $e_i$ a function in the finite family mentioned above.
We have two cases:

\medskip

\noindent {\em Case 1:} $\{a_1,e_1,\dots,e_\nu\}$ is not independent in $(T_M\Mm)^*$, that is $a_1$ is equal to a linear combination of $(e_1,\dots,e_\nu)$.   The set of $\bar{x} \in \U^*(\kappa,\eps^\nu)''$ satisfying this condition will be denoted by $\U^*(\kappa,\eps^\nu)''_1$.
Let $r=|a_1|=\omega(a_1)$, and $u_i+\imath v_i=\omega(e_i), \, i=1,\dots,\nu$. Then  $r$ is a linear function of $u_1,\dots,u_\nu$. Thus there is a constant $\delta$ such that $r \leq \delta \cdot \max\{|u_1|,\dots,|u_\nu|\}$. By assumption, we have $r < \delta\eps\sqrt{\Aa(M)}$.

We now observe that, up to some renumbering, the family $(a_1,e_1,\dots,e_{\nu-1})$ is independent in $(T_M\Mm)^*$.
Thus we have $\bar{x} \in \U^*(\delta\eps,\eps^{\nu-1})$.
Since there are only finitely many choices for the homology class of $e_i$ in $H_1(M,\Sig,\Z)$, it follows that $\delta$ belongs to a finite set.
Hence there is a constant $\delta_0>0$ such that $\U^*(\kappa,\eps^\nu)''_1 \subset \U^*(\delta_0\eps,\eps^{\nu-1})$ for $\eps>0$ small enough.
Using Corollary~\ref{cor:estimate:eps:1}, we conclude that there are some constants $\eps''_1>0$ and $K''_1 >0$ such that for any $0< \eps < \eps''_1$, we have
$$
    \int_{\S^1\times\U^*(\kappa,\eps^\nu)''_1}e^{-\Aa\circ\widetilde{\Phi}} \widetilde{\Phi}^* d\vol \leq \int_{\S^1\times\U^*(\delta_0\eps,\eps^{\nu-1})} e^{-\Aa\circ\widetilde{\Phi}} \widetilde{\Phi}^* d\vol < K''_1\eps^{2\nu}.
$$
\medskip

\noindent {\em Case 2:} $\{a_1,e_1,\dots,e_\nu\}$ is  independent in $(T_M\Mm)^*$.  The set of $\bar{x} \in \U^*(\kappa,\eps^\nu)''$ satisfying this condition will be denoted by $\U^*(\kappa,\eps^\nu)''_2$. This case will be handled in a similar manner to Proposition~\ref{prop:estimate:kap:eps:1}.

Let $s=d-\nu-1$.
We first add to the family $\{a_1,e_1,\dots,e_\nu\}$ $s$ saddle connections in the family $\{b_1,\dots,b_k\}$ to get a basis $\Bc'$ of $(T_M\Mm)^*$.
Without loss of generality, we can assume that $\Bc'=\{a_1,e_1,\dots,e_\nu,b_1,\dots,b_{s}\}$.
Using $\Bc'$, we define an $\R$-linear bijection $\phi: \R \times\R^{2(d-1)} \ra V_\alpha$, and set
$$
\begin{array}{cccc}
\Phi: & \S^1\times \phi^{-1}(\U_{\G,\alpha})& \ra & \Mm\\
      & (\theta, w) & \mapsto & e^{\imath\theta}\cdot\widetilde{\Phi}\circ\phi(w)\\
\end{array}
$$
The map $\phi$ is defined in such a way that if $\allowbreak w=(r,u_1,v_1,\dots,u_\nu,v_\nu,  x_1,y_1,\dots,x_s,y_s) \in \phi^{-1}(\U_{\G,\alpha})$,
and  $M=(X,\omega)=\widetilde{\Phi}\circ\phi(w)$, then $\omega(a_1)=r, \omega(e_i)=u_i+\imath v_i, \, i=1,\dots,\nu,$ and $\omega(b_j)=x_j+\imath y_j, \; j=1,\dots,s$.
Let $\mu_{2d-1}$ denote  the standard Lebesgue measure on $\R\times\R^{2(d-1)}$, that is
$$
    d\mu_{2d-1}=drdu_1dv_1\dots du_\nu dv_\nu dx_1dy_1\dots dx_{s}dy_{s}.
$$
By an elementary computation, we get
$$
\Phi^*d\vol=\lbd r d\theta d\mu_{2d-1}
$$
where $\lbd$ is a constant determined by the basis $\Bc'$.

Since for any $j \in \{1,\dots,k\}$ the height $y_j$ of $C_j$ is a linear function of $(v_1,\dots,v_{\nu},y_1,\dots,y_{s})$, we have
$$
    \Aa\circ\Phi=r\cdot \left(f(v_1,\dots,v_\nu)+\sum_{j=1}^{s}\xi_jy_j\right)
$$
where $f$ is a linear function determined by $(\G,\alpha)$ and the homology classes of $e_1,\dots,e_{\nu}$ in $H_1(M,\Sig,\Z)$.
By the same argument as in Proposition~\ref{prop:estimate:kap:eps:1}, we see that if $\eps>0$ is small enough then we can always assume that $\xi_{s}\neq 0$.
We will now use the change of variables
$$
(r,u_1,v_1,\dots,u_\nu,v_\nu, x_1,y_1,\dots, x_s,y_{s}) \mapsto (r,u_1,v_1,\dots,u_\nu,v_\nu,x_1,y_1,\dots,x_{s-1},y_{s-1},x_{s},\Aa).
$$
A direct computation shows
$$
  d\mu_{2d-1}=\frac{1}{\xi_s r}drdu_1dv_1\dots,du_\nu dv_\nu dx_1dy_1\dots dx_{s-1}dy_{s-1}dx_sd\Aa.
$$
Hence
$$
\Phi^*d\vol=\frac{\lbd}{\xi_s}d\theta dr du_1 dv_1\dots du_{\nu}dv_\nu dx_1dy_1\dots dx_{s-1}dy_{s-1}dx_s d\Aa.
$$
Let $\UU''_2(\kappa,\eps^\nu) \subset \R\times\R^{2(d-1)}$ be the domain defined by the following conditions:
\begin{itemize}
     \item [$\bullet$] $\Aa\circ\Phi >0$,


     \item[$\bullet$] $0 < r  < \kappa\sqrt{\Aa\circ\Phi}$,

     \item[$\bullet$] $\max\{|u_i|,|v_i|\}< \eps\sqrt{\Aa\circ\Phi},\, i=1,\dots,\nu,$

     \item[$\bullet$] $0\leq x_j < \lbd_jr, \text{ and } 0< y_j, \, j=1,\dots,s$.
\end{itemize}
where $\lbd_jr$ is the circumference of $C_j$ (the constant $\lbd_j$ is determined by $(\G,\alpha)$).
By assumption, we have $\bar{x} \in \phi( \UU''_2(\kappa,\eps^\nu))$.
Using
$$
    \Aa(M)\geq \frac{\Aa(M)}{2}+ \frac{1}{2}\sum_{j=1}^{s-1} \Aa(C_j)=\frac{\Aa(M)}{2} + \frac{1}{2}r\sum_{j=1}^{s-1} \lbd_j y_j
$$
we get
\begin{eqnarray*}
    \int_{\S^1\times \UU''_2(\kappa,\eps^\nu)}e^{-\Aa\circ\Phi}\Phi^*d\vol & \leq  & \frac{2\pi\lbd}{\xi_s} \int_{0}^{+\infty} \left(\int_0^{\kappa\sqrt{\Aa}}dr \prod_{i=1}^\nu \left(\int_{-\eps\sqrt{\Aa}}^{\eps\sqrt{\Aa}}  du_i \int_{-\eps\sqrt{\Aa}}^{\eps\sqrt{\Aa}} dv_i \right) \times \right. \\
                                                            &      & \times\left. \prod_{j=1}^{s-1} \left( \int_0^{\lbd_jr} dx_j \int_0^{+\infty}e^{-\lbd_jry_j/2}dy_j \right) \int_0^{\lbd_sr}dx_s \right)e^{-\Aa/2}d\Aa\\
                                                            & \leq  & \tilde{K}\eps^{2\nu}.
\end{eqnarray*}
where $\tilde{K}$ is a constant depending only on $(\kappa,\nu,\lbd,\lbd_s,\xi_s)$.
Since $\U^*(\kappa,\eps^\nu)''_2$ is covered by a finite family of subsets of $V_\alpha$ of the form $\phi(\UU''_2(\kappa,\eps^\nu))$, we draw that there exists a constant $K''_2>0$ such that if $\eps>0$ is small enough then
$$
    \int_{\S^1\times\U^*(\kappa,\eps^\nu)''_2} e^{-\Aa\circ\widetilde{\Phi}}\widetilde{\Phi}^*d\vol < K''_2 \eps^{2\nu}.
$$
Thus there are some constants $K''>0$ and $\eps_0>0$ such that if $0< \eps <\eps_0$ then
$$
    \Ic''=\int_{\S^1\times\U^*(\kappa,\eps^\nu)''_1}e^{-\Aa\circ\widetilde{\Phi}}\widetilde{\Phi}^*d\vol + \int_{\S^1\times\U^*(\kappa,\eps^\nu)''_2} e^{-\Aa\circ\widetilde{\Phi}}\widetilde{\Phi}^*d\vol < K''\eps^{2\nu}.
$$
The proof of the proposition is now complete.
\end{proof}

\subsection{Proof of Theorem~\ref{thm:estimate:eps:rk1}}\label{sec:prf:main:thm}
\begin{proof}
 It is a well known result by Masur-Smillie that for each stratum, there exists a constant $\kappa_0>0$ such that every surface in that stratum has a simple closed geodesic of length bounded by $\kappa_0\sqrt{\Aa}$. Some explicit estimates of the constant $\kappa_0$ are obtained  by Vorobets in \cite{Vor03}.

 Consider now a surface $M \in \Mm(\eps^\nu)$. Since the set of surfaces admitting a non-stable cylinder decomposition has zero measure in $\Mm$ (see Lemma~\ref{lm:stable:full:measure}), we can assume that every cylinder decomposition of $M$ is $\Mm$-stable. By the result of Masur-Smillie, $M$ admits an $\Mm$-stable cylinder decomposition in a direction $\theta$ such that $\Sys(M,\theta) < \kappa_0\sqrt{\Aa}$. Let $\G$ be the diagram of this cylinder decomposition, and $\Alp_\G^*$ be the finite set in Theorem~\ref{thm:good:subsp:finite:G}.
 By definition, there exists $\alpha \in \Alp_\G^*$ such that $M$ is contained in the image of $\S^1\times\U^*_{\G,\alpha}(\kappa_0,\eps^\nu)'$ by $\widetilde{\Phi}_{\G,\alpha}$.
 Since the set of cylinder diagrams for each stratum is finite, and for each cylinder diagram the set $\Alp_\G^*$ is finite by Theorem~\ref{thm:good:subsp:finite:G}, from Proposition~\ref{prop:estimate:kap:eps:2} we get
 $$
 \int_{\Mm(\eps^\nu)} e^{-\Aa}d\vol < K_0\eps^{2\nu}.
 $$
 for some constant $K_0>0$ if $\eps>0$ is small enough.

 For the last assertion, consider the following coordinate change $\Mm \ra \R_{>0}\times \Mm_1, M \mapsto (t,\frac{1}{t}\cdot M)$, with $t=\sqrt{\Aa(M)}$.
 It follows
 $$
 K_0\eps^{2\nu} > \int_{\Mm(\eps^\nu)}e^{-\Aa}d\vol=\int_0^{+\infty}t^{2d-1}e^{-t^2}dt\int_{\Mm_1(\eps^\nu)}d\vol_1=\frac{(d-1)!}{2}\vol_1(\Mm_1(\eps^\nu)),
 $$
 and we get the desired inequality.
\end{proof}


\appendix
\section{Strata of Abelian differentials}
It is well known that the space $\Hg$ is a complex algebraic orbifold of dimension $2g+n-1$, equipped with a natural volume form $\vol$ coming from the Lebesgue measure of $\C^{2g+n-1}$ via the period mapping. Let $\vol_1$ be the volume form on $\Hgi$ which is defined by the formula
$$
d\vol=d\vol_1 d\Aa,
$$
where $\Aa$ is the area function.
\begin{Definition}\label{def:inde:fam:1}
 A family  $\{\g_1,\dots,\g_m\}$  of $m$ saddle connections on a translation surface $M \in \Hg$ will be called  {\em independent} if $\{\g_1,\dots,\g_m \}$ is linearly independent in 
 $H_1(M,\Sig,\R)$, where $\Sig$ is the set of singularities of $M$.
\end{Definition}
By convention, two saddle connections $\g,\g'$ are said to be {\em disjoint} if we have $\inter(\g) \cap \inter(\g')=\vide$.

\medskip 

Given $\ul{\eps}=(\eps_1,\dots,\eps_m) \in (\R_{>0})^m$, we denote by $\Hgie$ the subset of $\Hgi$ consisting of surfaces $M\in \Hgi$ on which there exists an ordered independent family of $m$ disjoint saddle connections $\{\g_1,\dots,\g_m\}$ such that
$$
|\g_j| < \eps_j, \; j=1,\dots,m.
$$
\noindent When $\eps_j$ are small enough, such a family  exists only if $m < \dim_\C\Hg =2g+n-1$. In this section, we will give a proof the following
\begin{Theorem}\label{thm:sm:SC:vol}
For any $m< \dim_\C\Hg$, there exists a constant $K=K(\underline{k},m)$ such that
$$
\vol_1(\Hgie) < K\eps_1^2\dots\eps_m^2.
$$
\end{Theorem}

\subsection{Translation surface with marked saddle connections}\label{sect:TS:w:mSC}
\subsubsection{Local chart and volume form}\label{sect:chart:vol}
\begin{Definition}\label{def:TS:w:mSC}
Let us denote by $\Hgm$ the moduli space of pairs $(M, \{\g_1,\dots,\g_m\})$, where
\begin{itemize}
 \item $M$ is a translation surface in $\Hg$,
 \item $\{\g_1,\dots,\g_m\}$ is an independent ordered family of disjoint saddle connections in $M$
\end{itemize}
\end{Definition}


There is a forgetful map $F: \Hgm \ra \Hg$ which maps  a pair $(M,\{\g_1,\dots,\g_m\})$ to  $M$. If $M'\in \Hg$ is close enough to $M$, then there exist $m$ saddle connections $\g'_1,\dots,\g'_m$ on $M'$ such that $(M,\{\g_1,\dots,\g_m\})$ and $(M',\{\g'_1,\dots,\g'_m\})$ have the same topological type, that is there exists a homeomorphism $\varphi : M \ra M'$
 which realizes a bijection between  the sets of singularities  of $M$ and $M'$ preserving the orders, and satisfies $\varphi(\g_i)=\g'_i$.
Therefore, we can identify a neighborhood of $(M, \{\g_1,\dots,\g_m\})$ in $\Hgm$ with a neighborhood of $M$ in $\Hg$. We can  use period mappings to define local charts for $\Hgm$,  and pullback the volume form $\vol$ of $\Hg$ to $\Hgm$.

\subsubsection{Energy functions}
 For every $\ul{\eps}=(\eps_1,\dots,\eps_m) \in (\R_{>0})^m$, we define a function $\mathfrak{F}_{\ul{\eps}}: \Hgm \rightarrow \R$  by
$$
\mathfrak{F}_{\ul{\eps}}: (M,\{\g_1,\dots,\g_m\}) \mapsto \exp(-\sum_{j=1}^m\frac{ |\g_j|^2}{\eps_j^2}-\Aa(M)).
$$
Theorem~\ref{thm:sm:SC:vol} is a consequence of the following

\begin{Theorem}\label{thm:int:fin:EF}
There exists a constant $K=K(\ul{k},m)$ such that
$$
\int_{\Hgm} \mathfrak{F}_{\ul{\eps}} d\vol < K \eps_1^2\dots\eps_m^2.
$$
\end{Theorem}
The  proof of this theorem will be given in Section~\ref{sec:prf:thm:int:fin:EF}

\subsection{Proof of Theorem~\ref{thm:sm:SC:vol}}\label{sec:prf:abeldiff}
\begin{proof}
Let us now give the proof of \ref{thm:sm:SC:vol} using Theorem~\ref{thm:int:fin:EF}.
Define $\Hge=\R^*_+\cdot\Hgie$, where the action of $\R^*_+$ is given by $t\cdot(X,\omega)=(X,t\omega), \; t \in \R^*_+$.
Let $\Hgme \subset \Hgm$ be the pre-image of $\Hge$ under  the forgetful map $F: \Hgm \ra \Hg$.
Theorem~\ref{thm:int:fin:EF} implies
\begin{equation}\label{ineq:sm:SC:vol:1}
 \int_{\Hgme}\mathfrak{F}_{\ul{\eps}}d\vol < \int_{\Hgm} \mathfrak{F}_{\ul{\eps}} d\vol< K\eps_1^2\dots\eps_m^2.
\end{equation}
Let us define a function $ f_{\ul{\eps}}: \Hg  \ra  \R $ as follows
$$
f_{\ul{\eps}} : M \mapsto e^{-\Aa(M)}\times \left( \sum_{ \begin{array}{c}\{\g_1,\dots,\g_m\} \\ \text{\tiny independent family of disjoint saddle connections}\end{array}} e^{-(\frac{|\g_1|^2}{\eps_1^2}+\dots+\frac{|\g_m|^2}{\eps_m^2})} \right).
$$
By definition, we have
$$
\int_{\Hgme}\mathfrak{F}_{\ul{\eps}}d\vol=\int_{\Hge}f_{\ul{\eps}}d\vol.
$$
But if $M\in\Hge$, then by definition there exists at least an independent family $\{\g_1,\dots,\g_m\}$ in $M$ such that $|\g_j|^2/\eps_j^2 < \Aa(M), \, j=1,\dots,m$, hence
$$
e^{-(\frac{|\g_1|^2}{\eps_1^2}+\dots+\frac{|\g_m|^2}{\eps_m^2})} > e^{-m\Aa(M)}.
$$
Thus  we have $f_{\ul{\eps}} (.) > e^{-(m+1)\Aa(.)}$ on  $\Hge$.  The inequality \eqref{ineq:sm:SC:vol:1} implies
\begin{equation}\label{ineq:sm:SC:vol:2}
\int_{\Hge}e^{-(m+1)\Aa}d\vol < K \eps_1^2\dots \eps_m^2.
\end{equation}
Recall that the space $\Hg$ can be locally identified with an open of $\R^{2N}, N= 2g+n-1$, and $\vol$  with the Lebesgue measure. By this identification, $\Hgi$ corresponds to the hypersurface $\mathbf{Q}_1$ defined by the equation $\Aa=1$. Let us make the following change of coordinates
$$
\begin{array}{cccl}
 \Phi : & \mathbf{Q}_1\times (0,+\infty) & \ra & \R^{2N}\\
    & (v,t)                           & \mapsto & t\cdot v.
\end{array}
$$
Note that we have $\Phi^*\Aa(v,t)=t^2$. Using the relation $d\vol=d\Aa d\vol_1$ on $\mathbf{Q}_1$, one can easily check that
$$
\Phi^*d\vol=2t^{2N-1}dtd\vol_1.
$$
Therefore,
$$
\int_{\Hge}e^{-(m+1)\Aa}d\vol=\int_{\Hgie}\left(\int_0^{+\infty}2t^{2N-1}e^{-(m+1)t^2}dt\right)d\vol_1= \frac{(N-1)!}{(m+1)^N}\vol_1(\Hgie).
$$
It follows immediately from (\ref{ineq:sm:SC:vol:2}) that
$$
\vol_1(\Hgie)< K\eps_1^2\dots\eps_m^2.
$$
Theorem~\ref{thm:sm:SC:vol} is then proved.
\end{proof}

\subsection{Triangulations and local charts}\label{sect:Triang}
Let $(M,\{\g_1,\dots,\g_m\})$ be a point in $\Hgm$. Let $\Sig$ denote the set of singularities of $M$.
There always exists a geodesic triangulation  $\T$ of $M$ such that $\T^{(0)}=\Sig$, and  $\{\g_1,\dots,\g_m\}$ is included in the $1$-skeleton $\T^{(1)}$ of $\T$.
We will call such a triangulation an {\em admissible triangulation}.
Let $N_1, N_2$ denote the number of edges and  of triangles of $\T$  respectively.  Note that we must have $ N_1=3(2g+n-2) \text{ and } N_2=2(2g+n-2)$.

Recall that $M$ is defined by a holomorphic $1$-form $\omega$ on a Riemann surface $X$. We can associate to each oriented edge $e \in \T^{(1)}$ a complex number $Z(e)=\omega(e)$. Hence, the triangulation $\T$ provides us with a vector $Z\in \C^{N_1}$ whose coordinates satisfy the following condition: if $e_i,e_j,e_k$ are three edges of a triangle of $\T$, then
\begin{equation}\label{TriaEq}
\pm Z(e_i)\pm Z(e_j) \pm Z(e_k)=0.
\end{equation}
The signs of $Z(e_i),Z(e_j),Z(e_k)$ depend on their orientation. We have $N_2$ equations of type (\ref{TriaEq}), each of which corresponds to a triangle in  $\T^{(2)}$.

Let $\ST$ denote the system consisting of those linear equations, and $\VT$ be the subspace of solutions of $\ST$ in $\C^{N_1}$.
Let $Z'$ be a vector in $\VT$. If $Z'$ is close enough to $Z$ then there exists
\begin{itemize}
\item[$\bullet$] an element $(M',\{\g'_1,\dots,\g'_m\}) \in \Hgm$ close to $(M,\{\g_1,\dots,\g_m\})$,

\item[$\bullet$] an admissible triangulation $\T'$ of $M'$ such that $Z'$ is the vector associated to $\T'$,

\item[$\bullet$] a homeomorphism $\varphi: M\rightarrow M'$ which maps $\T$ to $\T'$.
\end{itemize}
\noindent
Therefore, we have a map $\DS{\Psi_\T: \mathcal{U} \rightarrow \Hgm}$, where $\mathcal{U}$ is an open subset of $\VT$ that contains $Z$.

\begin{Proposition} \label{prop:triang:loc:ch}
The subspace $\VT$ of $\C^{N_1}$ is of dimension $2g+n-1$.
For $\mathcal{U}\subset \VT$ small enough, $\Psi_\T$ is continuous and injective, it realizes an homeomorphism from $\mathcal{U}$ onto  its image.
\end{Proposition}

\subsection{Special triangulation}\label{sect:spec:triang}
\subsubsection{Construction}
Our goal in this section is to specify an open subset of $\VT$ on which the map $\Psi_\T$ is one-to-one. For this purpose, we first construct for each surface in an open dense subset (of full measure) of $\Hgm$  a particular triangulation which will be called {\em special triangulation}.

Recall that a geodesic ray emanating from  any point  in  $\Sig$ is called a {\em separatrix}.
Let $\Hgms$ denote the set of $(M,\{\g_1,\dots,\g_m\}) \in \Hgm$ satisfying
\begin{itemize}
 \item[(a)] none of the saddle connections $\{\g_1,\dots,\g_m\}$ is vertical,
 \item[(b)] every vertical separatrix  intersects the interior of one of the saddle connections $\{\g_1,\dots,\g_m\}$.
\end{itemize}

\begin{Lemma}\label{lm:dense:subset}
 $\Hgms$ is an open dense subset of full measure of $\Hgm$.
\end{Lemma}
\begin{proof}
It is not difficult to see that both  (a) and (b) are open conditions, thus $\Hgms$ is an open in $\Hgm$.

By a classical result (see for example \cite{KerMasSmi86}), the set of surfaces $M$ such that the vertical flow is minimal is of full measure in $\Hg$. On such a surface, every vertical separatrix  must intersect all saddle connections, since it is dense and not parallel to any saddle connection. Therefore, if $\{\g_1,\dots,\g_m\}$ is an independent  family of disjoint saddle connections on $M$, then $(M, \{\g_1,\dots,\g_m\}) \in \Hgms$. Since $\Hgm$ can be locally  identified with $\Hg$, it follows that $\Hgms$ is of full measure in $\Hgm$. Since $\vol$ is a volume form, a full measure subset of $\Hgm$ must be dense.
\end{proof}


Suppose from now on that $(M,\{\g_1,\dots,\g_m\})$ belongs to $\Hgms$.
We truncate all the vertical separatrices of $M$  at their first intersection with the set $\sqcup_{i=1}^m  \inter(\g_i)$.
Pick a saddle connection $\g$ in the family $\{\g_1,\dots,\g_m\}$. Let $A$ and $B$ denote the left and right endpoints of $\g$ respectively.  
We denote  the downward vertical separatrices that reach  $\inter(\g)$ by $\eta^+_1,\dots,\eta^+_{k}$  from the left to the right. 
The endpoints of $\eta^+_j$ will be denoted by $P_j$ and $Q_j$, where $P_j\in \Sig$ and $Q_j\in \inter(\g)$.  
Note that we may have $P_j=P_{j'}$ even if $j'\neq j$.

We will find a family of embedded triangles in $M$ with disjoint interior whose union contains all the segments $\eta^+_j$ by the following algorithm: 
let $j_0\in \{1,\dots,k\}$ be the smallest index such that $|\eta^+_{j_0}|=\min\{|\eta^+_{j}|, \; j=1,\dots,k\}$.
Using the developing map, we can realize  $\g$  and $\eta^+_{j_0}$ as two segments $\tilde{A}\tilde{B}$ and $\tilde{P}_{j_0}\tilde{Q}_{j_0}$  respectively in the plane $\R^2$, such that $\tilde{Q}_{j_0}\in \tilde{A}\tilde{B}$, and $\tilde{P}_{j_0}\tilde{Q}_{j_0}$ is vertical. Let $\tilde{\Delta}$ denote the triangle in $\R^2$ with vertices $\tilde{A},\tilde{B}, \tilde{P}_{j_0}$.
From the construction, the following claim is straightforward
\begin{Claim}
There exists a locally isometric map $\varphi: \tilde{\Delta} \ra M$ which satisfies:
\begin{itemize}
\item $\varphi(\tilde{A})=A, \varphi(\tilde{B})=B, \varphi(\tilde{P}_{j_0})=P_{j_0}$,
\item $\varphi(\inter(\tilde{\Delta}))$ is disjoint from the family $\{\g_1,\dots,\g_m\}$,
\item  the restriction of $\varphi$ to $\inter(\tilde{\Delta})$ is an embedding.
\end{itemize}
\end{Claim}

Let $\Delta$ be the image of $\tilde{\Delta}$ in $M$. Let $\g'$ and $\g''$ denote the saddle connections corresponding to the sides $AP_{j_0}$ and $BP_{j_0}$ of $\Delta$.
By construction,  $\eta^+_{j_0}$ is contained in the triangle $\Delta$, and all the other segments in the family $\{\eta^+_j, \, j\neq j_0\}$ intersect either $\g'$ or $\g''$. We can now apply the same arguments to $\g'$ and $\g''$ and continue the procedure until we get a family of embedded triangles that covers the union $\cup_{j=1}^k\eta^+_j$.

\begin{Remark}\label{rem:tri:proj}
If $\delta$ is a side of one of the triangles constructed above, and $\delta$ is not contained in the family $\{\g_1,\dots,\g_m\}$, then $\delta$ is one side of an embedded vertical  trapezoid ${\bf P}$ in $X$, whose vertical sides corresponds to two vertical segments $\eta^+_i, \eta^+_{j}$, and the opposite side of $\delta$ in ${\bf P}$ is a subsegment $Q_iQ_{j}$ of $\g$. We will call $Q_iQ_{j}$  the {\em projection} of $\delta$ on $\g$. Let  $\delta'$ be another side of one of the triangles we obtained, and $Q_{i'}Q_{j'}$ be its projection on $\g$. By construction, if $Q_iQ_j$ and  $Q_{i'}Q_{j'}$ intersect  then either $Q_iQ_j \subset Q_{i'}Q_{j'}$, or $Q_{i'}Q_{j'} \subset Q_iQ_j$.
\end{Remark}

By a symmetric procedure, we can find a family of embedded triangles with vertices in $\Sig$ and disjoint interiors that covers all the upward vertical separatrices that intersect $\inter(\g)$.


Applying this construction to all of the saddle connections in the family $\{\g_1,\dots,\g_m\}$, we get a collection of embedded triangles $\{\Delta_\alpha, \; \alpha \in I\}$ in $M$ with vertices in $\Sig$.

\begin{Claim}\label{clm:sp:triang}
 The family $\{\Delta_\alpha, \alpha \in I\}$ is a triangulation of $M$.
\end{Claim}
\begin{proof}
We first show that if $\Delta_{\alpha_1}$ and $\Delta_{\alpha_2}$ are two triangles in this family then $\inter(\Delta_{\alpha_1})\cap \inter(\Delta_{\alpha_2})= \varnothing$.
Assume that  $\inter(\Delta_{\alpha_1})\cap \inter(\Delta_{\alpha_2})\neq \varnothing$.
Since a triangle in the family $\{\Delta_\alpha, \alpha \in I\}$ cannot be included in the other one, there must exist a side $\delta_1$ of $\Delta_{\alpha_1}$, and a side $\delta_2$ of $\Delta_{\alpha_2}$ such that $\inter(\delta_1)\cap \inter(\delta_2)\neq \varnothing$.
By construction (see Remark~\ref{rem:tri:proj}) $\delta_i$ is one side of an embedded vertical trapezoid ${\bf P}_i$.
Since the endpoints of $\delta_i$ are singularities of $M$,  they  cannot be contained in $\inter({\bf P}_1)\cup \inter({\bf P}_2)$.
Recall that the opposite side of $\delta_i$ in ${\bf P}_i$ is a subsegment of a saddle connection in the family $\{\g_1,\dots,\g_m\}$, and $\delta_i$ and the vertical sides of ${\bf P}_i$ do not intersect the set $\sqcup_{1\leq i \leq m} \inter(\g_i)$.
Using these properties, one can easily check that if $\inter(\delta_1)\cap \inter(\delta_2)\neq \varnothing$ then we have a contradiction.

It remains to show that the union of all these triangles is $M$. Let $M'$ be the closure in $M$ of the set $M\setminus(\cup_{\alpha\in I}\Delta_\alpha)$. If $M'\neq \varnothing$, then $M'$ is a flat surface with piecewise geodesic boundary. Let $\Sig'$ be the finite subset of $M'$  arising from $\Sig$.
We can triangulate $M'$ by geodesic segments with endpoints in $\Sig'$.
 Let $\Delta$ be a triangle  in this triangulation. Consider the vertical rays starting from the vertices of $\Delta$. Observe  that one of those rays must intersects $\inter(\Delta)$.  This intersection is contained in  a vertical separatrix. But by construction, such a  subsegment must be contained in the union of $\{\Delta_\alpha,\; \alpha \in I\}$. Therefore we have a contradiction which proves the claim.
\end{proof}

We will call the triangulation constructed above the {\em special triangulation} of $(M,\{\g_1,\dots,\g_m\})$. By construction, this triangulation is unique. Note that such  a triangulation only exists for surfaces in $\Hgms$.

\subsection{Characterizing special triangulation}
 Let $\G$ be the dual graph of the special triangulation  $\T$ of $(M, \{\g_1,\dots,\g_m\})$, the vertices of $\G$ are the triangles in $\T^{(2)}$, and the (geometric) edges of $\G$ are the edges in $\T^{(1)}$. Remark that $\G$ is a trivalent graph, with $N_2$ vertices and $N_1$ edges. At each vertex of $\G$, we have a cyclic ordering on the set of edges incident to this vertex which is induced by the orientation of the surface.

For any $\g_i$, the union of the triangles that covers the vertical separatrices that intersect $\g_i$ from the upper side (resp. lower side) is dual to a tree in $\G$. We denote this tree by $\G_{i,+}$ (resp. $\G_{i,-}$), and choose its {\em root} to be the vertex dual to the triangle containing $\g_i$. By construction, all the trees $\G_{i,\pm}$ are disjoint, and any vertex of $\G$ belongs to one of those trees.  Note that if no vertical separatrix reaches $\inter(\g_i)$ from  the upper side (resp. from  the lower side), then the tree $\G_{i,+}$ (resp. $\G_{i,-}$) is empty.

Observe that the family of graphs $(\G,\{\G_{i,\veps}\})$ satisfies
\begin{itemize}
\item[a)] any vertex of $\G$ belongs to one of the trees $\G_{i,\veps}$,

\item[b)] for each $i\in \{1,\dots,m\}$ the roots of $\G_{i,+}$ and $\G_{i,-}$ are connected by an edge of $\G$.
\end{itemize}

\begin{Definition}\label{def:adm:fam:graphs}
We will call a trivalent graph with $N_1$ vertices and $N_2$ edges  equipped with a cyclic ordering on the set of edges incident to each vertex, together with $m$ subgraphs that are trees satisfying the conditions above an {\em admissible family of graphs}.
\end{Definition}

We number the edges of $\G$ in such a way that
\begin{itemize}
 \item[(i)] the first $m$ edges are the duals of $\{\g_1,\dots,\g_m\}$,

 \item[(ii)] if $e_{i_1}$ and $e_{i_2}$ belong to one of the trees $\G_{i,\veps}$, and $e_{i_1}$ is contained in the path from $e_{i_2}$ to the root, then  $i_1< i_2$,

 \item[(iii)] if $e_{i_1}$ belongs to one of the tree $\G_{i,\veps}$, and $e_{i_2}\not\in\{e_1,\dots,e_m\}$ does not belong to any tree, then $i_1 < i_2$.
\end{itemize}
We will call a numbering of the edges of $\G$ satisfying the conditions above a {\em compatible numbering} with respect to the trees $\{\G_{i,\pm}\}$.


Consider now the vector $Z=(z_1,\dots,z_{N_1})\in \C^{N_1}$, where $z_i=x_i+\imath y_i$ is the complex number  associated to $e_i$. 
By construction, we have $Z\in \VT$, the space of solutions of the system $\ST$.
Furthermore, by choosing a suitable orientation for the edges of $\T$, we can assume that
\begin{equation}\label{eq:real:part:pos}
 x_i:=\re(z_i) >0, \, i=1,\dots,N_1.
\end{equation}

We wish now to  describe the other properties of $Z$. Consider a triangle $\Delta$ in $\T^{(2)}$. Let $(e_{i_1},e_{i_2},e_{i_3})$ be the sides of $\Delta$  written in the cyclic order induced by the orientation of $M$, where $i_1=\min\{i_1,i_2,i_3\}$. Let $\G_\alpha$ be the tree  in the family $\{\G_{i,\pm}\}$ that contains the vertex of $\G$  dual to $\Delta$.

\begin{itemize}
\item[(a)] If $\Delta$ is the root of $\G_\alpha$, then $e_{i_1}\in \{e_1,\dots,e_m\}$. Otherwise $e_{i_1}$ is the unique edge of $\G_\alpha$ that is included in the path from $\Delta$ to the root. By construction, in both cases we always have
\begin{equation}\label{Ineq:hor:leng}
0<x_{i_2}<x_{i_1} \text{ and } 0<x_{i_3}<x_{i_1}.
\end{equation}
We will call $e_{i_1}$ the {\em base } of $\Delta$.

\item[(b)] We have $ \Aa(\Delta)=\frac{1}{2}\im(\bar{z}_{i_1}z_{i_2})$. Thus, $Z$ must satisfy
\begin{equation}\label{Ineq:Tri:Aa}
 \im(\bar{z}_{i_1}z_{i_2})=x_{i_1}y_{i_2}-x_{i_2}y_{i_1} >0.
\end{equation}

\item[(c)] Removing $e_{i_1}$ from the tree $\G_\alpha$, we get two subtrees $\G'_\alpha$ and $\G''_\alpha$, where $\G'_\alpha$ contains $\Delta$.
The union of the triangles of $\T$ dual to the vertices of $\G'_\alpha$ can be identified with a polygon $\mathbf{P}$ in $\R^2$ which has a side corresponding to $e_{i_1}$.
Without loss of generality, we can assume that $e_{i_1}$ is the lower side of $\mathbf{P}$.
The triangle $\Delta$ is identified with a triangle  in $\mathbf{P}$ whose vertices are denoted by $A,B,C$,  where $A$ and $B$ are the left  and right endpoints of $e_{i_1}$ respectively. By the convention on the cyclic ordering, we have $e_{i_2}=BC, e_{i_3}=CA$.

If $P$ is a vertex of $\mathbf{P}$ different from $A$ and $B$, we denote by $\hat{P}$ the intersection of the vertical (downward) ray from  $P$ with $\inter(AB)$, and by $h(P)$ the length of the segment $P\hat{P}$.
We will call $\hat{P}$ the projection of $P$ to $AB$. We denote the vertices of $\mathbf{P}$ whose projection to $AB$ is between $A$ and $\hat{C}$  by $D_1,\dots,D_r$,  and the vertices whose projection to $AB$ is between $\hat{C}$ and $B$  by $E_1,\dots,E_s$ from the left to the right.
By construction, we have
$$
h(C)<h(D_i), \; \forall i=1,\dots,r, \quad \text{ and } \quad h(C)\leq h(E_j), \; j=1,\dots,s.
$$
Note that the condition $h(C)<h(D_i)$ is equivalent to $ \overrightarrow{D_iC} \wedge \overrightarrow{AB} >0$.
As $\overrightarrow{D_iC} =\overrightarrow{D_iD_{i+1}} +\dots+\overrightarrow{D_rC}$, and the sides of $\mathbf{P}$ are the edges of $\T$, we can write $\overrightarrow{D_iC}$ as a linear function $f_i$ of $Z$. Consequently, the condition on $h(D_i)$ becomes
\begin{equation}\label{ineq:tree:cond:l}
 \im(\bar{z}_{i_1} f_i) <0, \; \forall i=1,\dots,r.
\end{equation}
\noindent Similarly, for $E_j, \; j=1,\dots,s$, we have $h(C) \leq h(E_j)$, and these conditions are equivalent two
\begin{equation}\label{ineq:tree:cond:r}
 \im(\bar{z}_{i_1}g_j) \geq 0, \; \forall j=1,\dots,s.
\end{equation}
\noindent where $g_j$ are some linear functions of $Z$. The functions $f_i,g_j$ are completely determined by the tree $\G'_\alpha$. Note also that, {\em a priori}, we cannot replace the inequality in (\ref{ineq:tree:cond:r}) by a strict one.
\end{itemize}

 Let $\Dc$ be the subset of $\VT$ which is defined by the inequations \eqref{eq:real:part:pos}, \eqref{Ineq:hor:leng}, \eqref{Ineq:Tri:Aa}, \eqref{ineq:tree:cond:l}, \eqref{ineq:tree:cond:r}. For every $Z=(z_1,\dots,z_{N_1})\in\Dc$, we define $\Psi(Z)$ to be the pair $(M,\{\g_1,\dots,\g_m\})$, where
\begin{itemize}
 \item[$\bullet$] $M$ is the  translation surface obtained by gluing the triangles constructed from the coordinates of $Z$ (using the graph $\G$).
  \item[$\bullet$] $\g_i$ is the saddle connection in $M$ corresponding to the edge of $\G$ which connects the roots of $\G_{i,+}$, and $\G_{i,-}$.
\end{itemize}
As a direction consequence of the construction of $\Psi$, we get
\begin{Proposition}\label{prop:psi:embed}
Let $\Dc_0$ be a component of  $\Dc$. If $\Psi(Z) \in \Hgm$ for some $Z\in \Dc_0$, then $\Psi(\Dc_0)\subset \Hgms$, and $\Psi$ realizes a homeomorphism from $\Dc_0$ onto its image. 
\end{Proposition}

Since the set of admissible families of graphs is clearly finite, we get
\begin{Corollary}\label{cor:good:chart:finite}
$\Hgms$ is covered by a finite family of  subsets of the form $\Psi(\Dc_0)$. 
\end{Corollary}

\subsection{Primary and auxiliary families of indices}
Let $(\G,\{\G_{i,\pm}\})$ be an admissible family of graphs, where $\G$ is the dual graph of a special triangulation $\T$ of some element $(M,\{\g_1,\dots,\g_m\})\in \Hgms$.  We denote by $\SG$ the systems of linear equations of type (\ref{TriaEq}) associated to $\G$, and by $\VG$ the subspace of $\C^{N_1}$ consisting of solution of $\SG$.

In what follows, we will say that a family of coordinates  $\{z_i, i \in I\}$ is {\em dependent } with respect to $\SG$ if there exists a vector $(\lbd_i)_{i\in I} \in \R^{I}$ such that $\sum_{i\in I} \lbd_i z_i=0$, for all $(z_1,\dots,z_{N_1})\in \VG$. The family $\{z_i, \, i \in I\}$ is said to be {\em  independent} with respect to  $\SG$, if such a vector does not exist. Equivalently, the family of coordinates $\{z_i, \; i \in I\}$ is independent if the restriction of the projection $(z_1,\dots,z_{N_1}) \in \C^{N_1} \mapsto (z_i)_{i\in I} \in \C^{|I|}$ to $\VG$ is surjective. 


\begin{Remark}\hfill
\begin{itemize}
 \item Since the coefficients of the equations of $\SG$ are in $\{0,\pm 1\}$, if the family $\{z_i, \, i \in I\}$ is dependent, we can choose the vector $(\lbd_i)_{i\in I}$ such that $\lbd_i\in \Z$, for all $i\in I$.
 \item Let $e_i$ denote the edge of $\T$ which corresponds to $z_i$. A family $I$ is independent with respect to $\SG$ if and only if the family $\{e_i, \; i \in I\}$ is independent in $H_1(M,\Sig, \Z)$.
\end{itemize}
\end{Remark}

\begin{Definition}\label{def:Prim:Fam}
 The {\em primary family}  of indices of $(\G,\{\G_{i,\pm}\})$ is the unique ordered subset $(i_1,\dots,i_{2g+n-1})$ of $\{1,\dots,N_1\}$ that satisfies the following properties
\begin{itemize}
 \item[$\bullet$] $i_k<i_{k+1}$,
 \item[$\bullet$] $i_1,\dots,i_{2g+n-1}$ is an independent family of indices with respect to $\SG$,
 \item[$\bullet$] For all $k\in \{1,\dots, 2g+n-1\}$, if $i<i_k$ then the family $(i_1,\dots,i_{k-1},i)$ is dependent  with respect to $\SG$.
\end{itemize}
\end{Definition}

\begin{Remark}
 The primary family of indices can be found inductively  by the following algorithm: first we take $i_1=1,\dots,i_m=m$. Recall that $e_1=\g_1,\dots,e_m=\g_m$, and by assumption $(\g_1,\dots,\g_m)$ is independent in $H_1(M,\Sig,\Z)$, hence the family $\{i_1,\dots,i_m\}$ is independent with respect to $\SG$. Assume that we already have an independent family $(i_1,\dots,i_k)$, then $i_{k+1}$ is the smallest index such that $(i_1,\dots,i_k,i_{k+1})$ is independent.
\end{Remark}

Let us denote by $I$  the primary family of index of $(\G, \{\G_{i,\pm}\})$.
\begin{Definition}\label{def:aux:fam}
 An {\em auxiliary family} of $I$ is an ordered family of indices $J=(j_{m+1},\dots,j_{2g+n-1})$ which satisfies the following condition: for  any $k\in\{m+1,\dots,2g+n-1\}$, $e_{j_k}$ is the base of one of the triangles in $\T$ that contains $e_{i_k}$. In particular, we have $j_k<i_k$.
\end{Definition}

\begin{Remark}
 By definition, we have $j_k<i_k$, thus $(i_1,\dots,i_{k-1},j_k)$ is a dependent family. Since the family $(i_1,\dots,i_{k-1})$ is independent, it follows that $z_{j_k}$ is a linear function  of $(z_{i_1},\dots,z_{i_{k-1}})$ in $\VG$.
\end{Remark}
An auxiliary family of $I$ can be found as follows: for each $i_k \in I, \, k>m$, consider the corresponding edge $e_{i_k}$ of $\G$. We have two cases:
\begin{itemize}
 \item[$\bullet$] $e_{i_k}$ belongs to some tree $\G_\alpha$. In this case, let $\Delta_k$ be the triangle of $\T$ dual to the endpoint of $e_{i_k}$ which is closer to the root of $\G_\alpha$. Let $e_{i'_k}$ and $e_{i''_k}$ be the other sides of $\Delta_k$ where $i'_k>i''_k$. We have $i_k \neq \min\{i_k,i'_k,i''_k\}$ by the definition of compatible numbering. Thus $i''_k=\min\{i_k,i'_k,i''_k\}$, which implies that $e_{i''_k}$ is the base of $\Delta_k$, and we can choose $j_k$ to be $i''_k$.

\item[$\bullet$] $e_{i_k}$ does not belong to any of the trees $\G_\alpha$. Let $\Delta_k$ be one of the two triangles in $\T$ that contains $e_{i_k}$. Again let $e_{i'_k}$ and $e_{i''_k}$ be the other sides of $\Delta_k$ where $i'_k> i''_k$. By the definition of compatible numbering, we also have $i_k \neq \min\{i_k,i'_k,i''_k\}$, hence we can choose $j_k$ to be $i''_k$.
\end{itemize}

 \subsection{Proof of Theorem~\ref{thm:int:fin:EF}}\label{sec:prf:thm:int:fin:EF}
 \begin{proof}
 Fix an admissible family of graphs $(\G,\{\G_{i,\pm}\})$ together with a compatible numbering of the edges of $\G$. Let $I=(i_1,\dots,i_{2g+n-1})$ be the primary family of indices of $(\G,\{\G_{i,\pm}\})$, and $\allowbreak J=(j_{m+1},\dots,j_{2g+n-1})$ an auxiliary family of indices of $I$. Since $\card\{I\}=\dim_\C \VG$, there exists a linear isomorphism with rational coefficients
 $$
 \begin{array}{cccc}
 \mathrm{F}: & \C^{2g+n-1} & \ra &  \VG\\
             & (z_1,\dots,z_{2g+n-1}) & \mapsto  & (f_1(z_1,\dots,z_{2g+n-1)}),\dots,f_{N_1}(z_1,\dots,z_{2g+n-1}))\\
 \end{array}
 $$
 which satisfies $f_{i_k}=z_k, \, \forall i_k \in   I$. To simplify the notation, we write $f_k$ instead of $f_{j_k}$, for every $j_k\in J$.

 Let $\Dc$ be the domain of $\VG$ which is defined by the inequations \eqref{eq:real:part:pos}, \eqref{Ineq:hor:leng}, \eqref{Ineq:Tri:Aa}, \eqref{ineq:tree:cond:l}, \eqref{ineq:tree:cond:r}. By a slight abuse of notation, we will denote also by $\Dc$ the corresponding domain  in $\C^{2g+n-1}$ (via $\mathrm{F}$). For every $(z_1,\dots,z_{2g+n-1})\in \Dc$, let $\Aa((z_1,\dots,z_{2g+n-1}))$ denote the area of the surface constructed from $F((z_1,\dots,z_{2g+n-1}))$.

 Let us write $z_i=x_i+\imath y_i$, with $x_i,y_i\in\R, \, i=1,\dots,2g+n-1$, and $f_k=a_k+\imath b_k, \, a_k,b_k\in \R, \, k=m+1,\dots,2g+n-1$. Remark that the coefficients of $f_k$ are real, therefore $a_k$ is a real linear function of $(x_1,\dots,x_{k-1})$, and $b_k$ is a real linear function of $(y_1,\dots,y_{k-1})$.

 To prove the theorem, by Corollary~\ref{cor:good:chart:finite}, is is enough to show that
 \begin{equation}\label{ineq:int:fin:1}
 \int_{\Dc} e^{-(|z_1|^2/\eps_1^2+\dots+|z_m|^2/\eps_m^2)}e^{-\Aa}dx_1dy_1\dots dx_{2g+n-1}dy_{2g+n-1} < K\eps_1^2\dots\eps_m^2.
 \end{equation}
 To prove the inequality (\ref{ineq:int:fin:1}) we will make use of the conditions (\ref{Ineq:hor:leng}) and (\ref{Ineq:Tri:Aa}). We first observe that the inequation (\ref{Ineq:hor:leng}) implies that for any $(x_1+\imath y_1,\dots,x_{2g+n-1}+\imath y_{2g+n-1})\in \Dc$, we have
 \begin{equation}\label{ineq:hor:leng:b}
 0<x_k<a_k, \;  k=m+1,\dots,2g+n-1.
 \end{equation}
 Set
  $$
 \Dc^*= \{(\underline{z},\underline{x}), \,\ul{z}\in \C^m, \underline{x}\in (\R_{>0})^{2g+n-1-m}, \; (\underline{z},\underline{x}) \text{  satisfies  (\ref{ineq:hor:leng:b})} \}\subset \C^m\times\R_{>0}^{2g+n-1-m}.
 $$
 \noindent For any  $(\underline{z},\underline{x})\in \Dc^*$, set
 $$
 \mathcal{Y}(\underline{z},\underline{x})=\{\underline{y}=(y_{m+1},\dots,y_{2g+n-1}) \in \R^{2g+n-1-m}:  \; (\underline{z},\underline{x}+\imath\underline{y})\in \Dc\}.
 $$
 \noindent The inequality (\ref{ineq:int:fin:1}) now becomes
 \begin{equation}\label{ineq:int:fin:2}
 \mathcal{I}:=\int_{\Dc^*} e^{-(|z_1|^2/\eps_1^2+\dots+|z_m|^2/\eps_m^2)}\left(\left(\int_{\mathcal{Y}(\underline{z},\underline{x})}e^{-\Aa} dy_{m+1}\dots dy_{2g+n-1}\right)d\mu_{2g+n-1-m}\right)d\mu_{2m} < K\eps_1^2\dots \eps_m^2.
 \end{equation}
 where $d\mu_{2g+n-1-m}=dx_{m+1}\dots dx_{2g+n-1}$ and $d\mu_{2m}=dx_1dy_1\dots dx_mdy_m$ are the Lebesgue measures of $\R^{2g+n-1-m}$ and $\C^m$ respectively.
 
 For every $k\in \{m+1,\dots,2g+n-1\}$, let $\Delta_k$ denote the triangle in $\T$ that contains $e_{i_k}$ and $e_{j_k}$.
 \begin{Claim}
  If $k\neq k'$, then $\Delta_k\neq \Delta_{k'}$.
 \end{Claim}
 \begin{proof}
  We can assume that $k<k'$. If $\Delta_k=\Delta_{k'}$ then $e_{i_k},e_{j_k}, e_{i_{k'}}$ are contained in the same triangle, which implies that $z_{i_{k'}}$ is a linear function of $(z_{i_k},z_{j_k})$ in $\VG$. By assumption, $z_{j_k}$ is a linear function of $(z_{i_1},\dots,z_{i_{k-1}})$, thus $z_{i_{k'}}$ is a linear function of $(z_{i_1},\dots,z_{i_k})$ which is impossible since $(i_1,\dots,i_{k'})$ is an independent family.
 \end{proof}

 Let $\eta_k$ be the area of $\Delta_k$. The previous claim implies immediately that
 \begin{equation}\label{ineq:area}
  \Aa > \sum_{k=m+1}^{2g+n-1} \eta_k.
 \end{equation}
 As functions on $\C^{2g+n-1}$,  $\eta_k$ are given by
 $$
 \eta_k = \pm \frac{1}{2}\im(f_k\bar{z}_k)= \pm \frac{1}{2}(a_ky_k-x_kb_k).
 $$
 Since $b_k$ is a linear function of $(y_1,\dots,y_{k-1})$, we have
 $$
 \frac{\partial \eta_k}{\partial y_j}=\left\{ \begin{array}{lc}
                                                 \frac{\pm 1}{2}x_k\frac{\partial b_k}{\partial y_j} & \hbox{ if $ j< k$},\\
                         \frac{\pm 1}{2}a_k & \hbox{ if $j=k$},\\
                          0 & \hbox{ if $j>k$},\\
                          \end{array} \right.
 $$
 hence
 $$
 d\eta_{m+1}\dots d\eta_{2g+n-1}=\frac{1}{2^{2g+n-1-m}}\left| \det\left( \begin{array}{cccc} \pm a_{m+1} & 0 & \dots & 0 \\ * & \pm a_{m+2} & \dots & 0 \\ \dots & \dots & \dots & \dots \\ * & * & \dots & \pm a_{2g+n-1} \end{array}\right)\right| dy_{m+1}\dots dy_{2g+n-1}.
 $$
 Therefore, we have
 $$
 dy_{m+1}\dots dy_{2g+n-1}=\frac{2^{2g+n-1-m}}{a_{m+1}\dots a_{2g+n-1}} d\eta_{m+1}\dots d\eta_{2g+n-1}.
 $$
 It follows from (\ref{ineq:area}) that
 $$
 \mathcal{I}(\underline{z},\underline{x}):=\int_{\mathcal{Y}(\underline{z},\underline{x})}e^{-\Aa}dy_{m+1}\dots dy_{2g+n-1} < \frac{2^{2g+n-1-m}}{a_{m+1}\dots a_{2g+n-1}}\int_{\mathcal{Y}(\underline{z},\underline{x})}e^{-(\eta_{m+1}+\dots+\eta_{2g+n-1})}d\eta_{m+1}\dots d\eta_{2g+n-1}.
 $$
 The conditions (\ref{Ineq:Tri:Aa}) mean that the functions $\eta_k$ are positive on $\mathcal{Y}(\underline{z},\underline{x})$. Thus we have
 $$
 \mathcal{I}(\ul{z},\ul{x}) < \frac{2^{2g+n-1-m}}{a_{m+1}\dots a_{2g+n-1}}\int_0^{+\infty}e^{-\eta_{m+1}}d\eta_{m+1}\dots\int_0^{+\infty}e^{-\eta_{2g+n-1}} d\eta_{2g+n-1}=\frac{2^{2g+n-1-m}}{a_{m+1}\dots a_{2g+n-1}},
 $$
 and
 \begin{equation}\label{ineq:int:ET:1}
 \mathcal{I} < 2^{2g+n-1-m}\int_{\Dc^*}\frac{e^{-(|z_1|^2/\eps_1^2+\dots+|z_m|^2/\eps_m^2)}}{a_{m+1}\dots a_{2g+n-1}}dx_{2g+n-1}\dots dx_{m+1}d\mu_{2m}.
 \end{equation}
 \noindent From the definition of $\Dc^*$, the right hand side of (\ref{ineq:int:ET:1}) is equal to
$$
 2^{2g+n-1-m} \int_{\C^m} e^{-(|z_1|^2/\eps_1^2+\dots+|z_m|^2/\eps_m^2)} \left(\int_0^{a_{m+1}} \dots \int_0^{a_{2g+m-1}}\frac{1}{a_{m+1} \dots a_{2g+n-1}}dx_{2g+n-1}\dots dx_{m+1}\right)         d\mu_{2m}.
$$ 
Since $a_{m+1},\dots,a_{k}$ do not depend on $x_k$, we have
  $$
  \int_0^{a_{m+1}}  \dots \int_0^{a_{2g+m-1}} \frac{1}{a_{m+1}\dots a_{2g+n-1}} dx_{2g+n-1}  \dots dx_{m+1}=1.
 $$
Thus
\begin{equation}\label{inq:int:ET:2}
 \mathcal{I} < 2^{2g+n-1-m}\int_{\C^m}e^{-(|z_1|^2/\eps^2_1+\dots+|z_m|^2/\eps^2_m)}dx_1dy_1\dots dx_mdy_m = 2^{2g+n-1-m}\pi^m \eps_1^2\dots\eps_m^2.
\end{equation}
\noindent The inequality (\ref{ineq:int:fin:2}) is then proved. Since the number of admissible families of graphs is finite, Theorem~\ref{thm:int:fin:EF} follows.
\end{proof}


\section{Strata of quadratic differentials}
\subsection{Statement of the result}\label{sec:QD:state:result}
Let $X$ be a compact Riemann surface, and $\phi$ a meromorphic quadratic differential on $X$ with simple poles. The quadratic deferential $\phi$ defines on $X$ a flat metric structure whose transition maps are of the form $z\mapsto \pm z +c, c$ is constant. Since all the poles of $\phi$ are simple, the area of this metric structure is finite, and its singularities are isolated cone points corresponding to the zeros and poles of $\phi$. The cone angle at a zero of order $k$ is $(k+2)\pi$, where a zero of order $-1$ is a simple pole.

In the case $\phi=\omega^2$, where $\omega$ is a holomorphic one-form, the flat metrics defined by $\phi$ and $\omega$ are the same. We will be only interested in quadratic differentials which are  not the square of any holomorphic one-form. In this case, the corresponding flat surface is also called a {\em half-translation surface}.
We denote by $\Qg$, $\ul{d}=(d_1,\dots,d_n), d_i\in \{-1,0,1,2,\dots\}$,  the set of pairs $(X,\phi)$ where $\phi$ has exactly $n$ zeros, with orders given by $\ul{d}$, and $\phi$ is not the square of a holomorphic one-form. By convention, a zero of order $0$ is a marked regular point. Remark that the genus of $X$ is determined by the formula
$$
d_1+\dots+d_n=4g-4.
$$
It is well known that $\Qg$ is a complex algebraic orbifold of dimension $2g+n-2$, which is equipped with a natural Lebesgue volume form $\vol$.

Let us recall briefly a way to define the local chart on a neighborhood of $(X,\phi) \in \Qg$. 
There exists a canonical (ramified) double cover $\piup: \hat{X}\ra X$ such that the quadratic differential $\hat{\phi}=\piup^*\phi$ is the square of a holomorphic one-form $\hat{\omega}$ on $\hat{X}$. We have an involution $\tp: \hat{X} \ra \hat{X}$ satisfying $x'=\tp(x)$ if and only if $\piup(x')=\piup(x)$, and $\tp^*\hat{\omega}=-\hat{\omega}$.
Let $\Sig$ be  the set of zeros and poles of $\phi$ and $\hat{\Sig}=\piup^{-1}(\Sig)\subset \hat{X}$.
We have a natural decomposition
$$
H_1(\hat{X},\hat{\Sig},\C)=H_1(\hat{X},\hat{\Sig},\C)^- \oplus H_1(\hat{X},\hat{\Sig},\C)^+,
$$
where $H_1(\hat{M},\hat{\Sig},\C)^\pm$ are the eigenspaces of $\tp$ corresponding to the eigenvalues $\pm 1$ respectively.

For a pair $(X',\phi')$ close to $(X,\phi)$, we have a pair  $(\hat{X}',\hat{\omega}')$ close to $(\hat{X},\hat{\omega})$ together with an involution $\tp': \hat{X}' \ra \hat{X}'$ (in the same homotopy class as $\tp$), and a ramified double cover $\piup': \hat{X}' \ra X'$. 
Let $\Sig'$ be the set of zeros and poles of $\phi'$ and $\hat{\Sig}'={\piup'}^{-1}(\Sig')$.
We can identify $H_1(\hat{X}',\hat{\Sig}',\Z)^-$ with $H_1(\hat{M},\hat{\Sig},\Z)^-$. 
Choose a basis $\{c_1,\dots,c_d\}$ of $H_1(\hat{X},\hat{S}; \Z)^-$, then the period mapping $\hat{\Phi}: (X',\phi') \mapsto (\int_{c_1}\hat{\omega}',\dots,\int_{c_d}\hat{\omega}')$ is a local chart for $\Qg$, in which the volume form $\vol$ is identified with the Lebesgue measure of $\C^d$.


Let $\Qgi$ denote the subset of $\Qg$ consisting of surface of area one, then we have a volume form $\vol_1$ on $\Qgi$ defined by
$$
d\vol=d\Aa d\vol_1,
$$
where $\Aa$ is the area function on $\Qg$.

To simplify the notation, we denote by  $M$ and $\hat{M}$ the flat surfaces defined by the pairs $(X,\phi)$ and $(\hat{X},\hat{\omega})$ respectively. 
A geodesic segment $\g$ on $M$ with endpoints in $\Sig$ which does not intersect $\Sig$ in the interior will be called a {\em  saddle connection}.
The pre-image of $\g$ in $\hat{M}$ consists of two geodesic segments $\hat{\g}^1,\hat{\g}^2$ with endpoints in $\hat{\Sig}$.
We choose the orientation of $\hat{\g}^1$ and $\hat{\g}^2$ so that $\tp(\hat{\g}^1)=-\hat{\g}^2$ and $\tp(\hat{\g}^2)=-\hat{\g}^1$.
We will denote by $\hat{\g}$ the element of  $H_1(\hat{M},\hat{\Sig},\Z)$ represented by $\hat{\g}^1+\hat{\g}^2$.
By definition, $\hat{\g}$ belongs to $H_1(\hat{M},\hat{\Sig},\C)^-$. 

\begin{Definition}\label{def:ind:fam:2}
A family  $\{\g_1,\dots,\g_m\}$ of $m$ saddle connections on $M$ is called {\em independent} if $\{\hat{\g}_1,\dots,\hat{\g}_m\}$ is independent in $H_1(\hat{M},\hat{\Sig},\C)^-$.
\end{Definition}

\noindent By convention two saddle connections $\g,\g'$ in $M$ are said to be disjoint if $\inter(\g)\cap\inter(\g')=\vide$.
For all $\ul{\eps}=(\eps_1,\dots,\eps_m)\in (\R_{>0})^m$, where $m < \dim_\C\Qg = 2g+n-2$, 
let us denote by $\Qgie$ the set of $M\in \Qgi$ such that there exists an ordered independent family of disjoint saddle connections $\{\g_1,\dots,\g_m\}$ on $M$ 
satisfying $|\g_j|< \eps_j, \; j=1,\dots,m$.
Similar to the case of Abelian differentials, we will prove
\begin{Theorem}\label{thm:sm:SC:QD}
For any $m < \dim_\C\Qg$, there exists a constant $K=K(\ul{d},m)$ such that for every $\ul{\eps}\in (\R_{>0})^m$, we have
$$
\vol_1(\Qgie)< K\eps_1^2\dots\eps_m^2.
$$
\end{Theorem}

\subsection{Half-translation surface with marked saddle connections}\label{sec:QD:w:marked:sc}
Let us denote by $\Qgm$ the set of pairs $(M,\{\g_1,\dots,\g_m\})$, where
\begin{itemize}
\item[$\bullet$] $M$ is an element of $\Qg$,
\item[$\bullet$] $\{\g_1,\dots,\g_m\}$ be an independent family of disjoint saddle connections in $M$.
\end{itemize}
There exists a natural projection $F: \Qgm \ra \Qg$ consisting of forgetting the marked saddle connections. 
This projection is locally homeomorphic, hence we can pullback the Lebesgue  measure $d\vol$ of $\Qg$ to $\Qgm$. 
We define a function $\mathfrak{F}_{\ul{\eps}} :\Qgm \ra \R$  as follows
$$
\mathfrak{F}_{\ul{\eps}}: (M,\{\g_1,\dots,\g_m\}) \mapsto \exp(-\sum_{j=1}^m\frac{|\g_j|}{\eps_j^2}-\Aa(M)).
$$
Theorem~\ref{thm:sm:SC:QD} is a direct consequence of the following theorem by the same arguments as in the proof of Theorem~\ref{thm:sm:SC:vol}
\begin{Theorem}\label{thm:int:fin:EF:QD}
There exists a constant $K=K(\ul{d},m)$ such that
$$\int_{\Qgm} \mathfrak{F}_{\ul{\eps}} d\mu < K \eps_1^2\dots\eps_m^2.$$
\end{Theorem}

\subsection{Symmetric special triangulation}
Consider an element $(M, \{\g_1,\dots,\g_m\})$ of $\Qgm$. Given $\theta\in \S^1$, the flow in direction $\theta$ is not well defined  on $M$, but the notion of ``{\em parallel}'' is. 
Let $\Qgms$ be the set of $(M,\{\g_1,\dots,\g_m\})\in \Qgm$ which satisfy the following condition: every vertical separatrix intersects the set $\cup_{1\leq i \leq m}\inter(\g_i)$ before reaching a singular point. The following lemma follows from the same argument as Lemma~\ref{lm:dense:subset}

\begin{Lemma}\label{lm:dense:subset:QD}
$\Qgms$ is an open dense subset of full measure of $\Qgm$.
\end{Lemma}

Suppose that $(M,\{\g_1,\dots,\g_m\})$ belongs to $\Qgms$, then $(\hat{M},\{\hat{\g}^1_1,\hat{\g}^2_1,\dots,\hat{\g}^1_m,\hat{\g}^2_m\})$ belongs to $\widetilde{\mathcal{H}}^{(2m)}(\ul{\hat{k}})^*$ (for some integral vector $\ul{\hat{k}}$). Let $\hT$ be the special triangulation of $\hat{M}$ with respect to the family $\{\hat{\g}^s_i\}$. 
Since we have $\tp^*\hat{\omega}=-\hat{\omega}$, it follows that the involution $\tp$ sends a separatrix indirection $(0,1)$ to a separatrix in direction $(0,-1)$ and vice versa.
By definition, the family of saddle connections $\{\hat{\g}^s_i\}$ is invariant under $\tp$. Therefore, $\tp$ induces an involution on  set of vertical separatrices joining every point in  $\hat{\Sig}$ to some points in  $\cup_{1 \leq i \leq m}(\inter(\hat{\g}^1_i)\cup \inter(\hat{\g}^2_i))$. As a consequence, the triangulation $\hT$ is invariant under $\tp$. 
It follows in particular that $\tp$ induces an involution on the set $\hT^{(1)}$.

Let $\hat{N}_1$ and $\hat{N}_2$ be the number of edges and of triangles in $\hT$ respectively. We choose the orientations of the edges of $\hT$ such that $\hat{\omega}(\tp(e))=-\hat{\omega}(e)$ for all $e \in \hT^{(1)}$.  
Consider a vector $Z\in \C^{\hat{N}_1}$ as a function from $\hT^{(1)}$ to $\C$, we have a system $\mathbf{S}_{\hT}$ of $\hat{N}_2$ linear  equations of the form (\ref{TriaEq}), each of which corresponding to a triangle in $\hT^{(2)}$. We add to this system the equations
\begin{equation}\label{SymEq}
Z(e')=Z(e)
\end{equation}
where $e,e'\in \hT^{(1)}$ such that $\tp(e)=-e'$. 

Let $\hat{\mathbf{S}}_{\hT}$ denote the resulting system, and $\hat{\mathrm{V}}_{\hT}\subset \C^{\hat{N}_1}$ denote the space of solutions of $\hat{\mathbf{S}}_{\hT}$. 
Clearly, the integrals of $\hat{\omega}$ along the edges of $\hT$ gives us a vector in $\hat{\mathrm{V}}_{\hT}$. 
Given $Z'\in \hat{\mathrm{V}}_{\hT}$ close to $Z$, we can construct a surface $\hat{M}'$, together with $2m$ marked saddle connections ${\hat{\g'}}^s_i$ and an involution $\tp':\hat{M}'\ra \hat{M}'$ satisfying $\tp'({\hat{\g'}}^1_i)=-{\hat{\g'}}^2_i$. 
Thus we have a continuous mapping $\Psi_{\hT}$ defined in a neighborhood of $Z$ in $\hat{\mathrm{V}}_{\hT}$ to $\Qgm$. It is not difficult to check that $\dim_\C\hat{\mathrm{V}}_{\hT}=\dim_\C\Qgm=2g+n-2$, and in local charts of $\Qgm$ defined by  period mappings, $\Psi_{\hT}$ is a linear isomorphism between complex vector spaces.

Let $\Dc$ be the domain in $\hat{\mathrm{V}}_{\hT}$ which is defined by the inequations (\ref{eq:real:part:pos}), (\ref{Ineq:hor:leng}),(\ref{Ineq:Tri:Aa}), (\ref{ineq:tree:cond:l}), (\ref{ineq:tree:cond:r}), and $\Dc_0$ be the component of $\Dc$ that contains $Z$. 
\begin{Proposition}\label{prop:psi:embed:QD}
The map $\Psi_{\hT}$ is well defined  and injective in $\Dc_0$, it realizes a homeomorphism between $\Dc_0$ and  its image in $\Qgm$.
\end{Proposition}


Let us now define
\begin{Definition}\label{def:sym:adm:graph}
Let $\G$ be a trivalent graph with $\hat{N}_2$ vertices and $\hat{N}_1$ edges. Let $\G^s_{i,\pm}, \, i\in \{1,\dots,m\}, \, s\in \{1,2\}$, be a family of disjoint subgraphs of $\G$ which are trees. We will say that $(\G,\{\G^s_{i,\pm}\})$ is a {\em symmetric admissible family of graphs} if $(\G,\{\G^s_{i,\pm}\})$ is an admissible family of graphs (see Definition~\ref{def:adm:fam:graphs}), and there exists an involution $\tp$ of $\G$ which satisfies $\tp(\G^1_{i,\veps})= \G^{2}_{i,-\veps}$. 
\end{Definition}

Let $\G$ be the dual graph of $\hT$. For each $\hat{\g}^s_i$, let $\G^s_{i,+}$ (resp. $\G^s_{i,-}$) be the tree in $\G$ which is dual to the union of triangles in $\hT^{(2)}$ that cover the vertical separatrices reaching $\inter(\hat{\g}^s_{i})$ from the upper side (resp. lower side). By construction, $(\G,\{\G^s_{i,\pm}\})$ is a symmetric admissible family of graphs. 
Since the set of symmetric admissible families of graphs is finite, we get

\begin{Proposition}~\label{prop:part:fin:QD}
There exists a partition of $\Qgms$ into finitely many subsets, each of which corresponds to a component of the subset  determined by the inequations (\ref{eq:real:part:pos}), (\ref{Ineq:hor:leng}),(\ref{Ineq:Tri:Aa}), (\ref{ineq:tree:cond:l}), (\ref{ineq:tree:cond:r}) in a vector subspace $\hat{\mathrm{V}}$ of $\C^{\hat{N}_1}$ of dimension $2g+n-2$. The space $\hat{\mathrm{V}}$ itself is determined by a symmetric admissible family of graphs.
\end{Proposition}


\subsection{Proof of Theorem~\ref{thm:int:fin:EF:QD}}
\begin{proof}
Let $(\G,\{\G^s_{i,\pm}\})$ be a symmetric admissible family of graphs. Choose a compatible numbering of the edges of $\G$ such that $e_i$ is the root of $\G^1_{i,+}, \, i=1,\dots,m$. Let $\mathbf{S}_\G$ denote the linear system associated to $(\G,\{\G^s_{i,\pm}\})$, and $\hat{\bf S}_\G$ the system obtained by adding to ${\bf S}_\G$ the equations of type (\ref{SymEq}). Let $\hat{\V}_\G$ be the space of solutions of $\hat{\mathbf{S}}_\G$ in $\C^{\hat{N}_1}$. 
Let $\Dc$ be the subset of $\hat{\V}_\G$ determined by the inequations (\ref{eq:real:part:pos}), (\ref{Ineq:hor:leng}), (\ref{Ineq:Tri:Aa}), (\ref{ineq:tree:cond:l}), (\ref{ineq:tree:cond:r}). 
By construction, there exists a map $\Psi :\Dc \rightarrow \Qgm$ which is homeomorphic onto its image.  
The pullback of the function $\mathfrak{F}_{\ul{\eps}}$  by $\Psi$ is given by
$$
\Psi^* \mathfrak{F}_{\ul{\eps}}(Z)=\exp(-(\frac{|z_1|^2}{\eps_1^2}+\dots+\frac{|z_m|^2}{\eps_m^2})-\Aa(\Psi(Z))).
$$
By Proposition~\ref{prop:part:fin:QD}, it suffices to show that
\begin{equation*}
\int_{\Dc}e^{-(|z_1|^2/\eps_1^2+\dots+|z_m|^2/\eps_m^2)-\Aa} d\mu < K\eps_1^2\dots\eps_m^2.
\end{equation*}
where $d\mu$ is the Lebesgue measure of $\hat{\V}_\G$.  
The remainder of the proof follows the same lines as the proof of Theorem~\ref{thm:int:fin:EF}. 
\end{proof}

\end{document}